%% file: RSV20_ArxivRevised.tex
\documentclass[12pt]{amsart}
\usepackage{amsmath,amsthm,amsfonts,amssymb,latexsym, mathabx}
\usepackage{enumerate}
\usepackage{color}
\usepackage{graphicx,scalerel}
\usepackage{lmodern, hyperref} 
\usepackage[dvipsnames]{xcolor}
\usepackage{multirow}
\usepackage{todonotes}

\begin{document}

\theoremstyle{plain}

\newtheorem{thm}{Theorem}[section]
\newtheorem{lem}[thm]{Lemma}
\newtheorem{pro}[thm]{Proposition}
\newtheorem{hyp}[thm]{Hypotheses}
\newtheorem{cor}[thm]{Corollary}
\newtheorem*{conj}{Conjecture}
\newtheorem*{thm*}{Theorem}

\theoremstyle{definition}
\newtheorem{rem}[thm]{Remark}
\theoremstyle{definition}
\newtheorem{defi}[thm]{Definition}
\theoremstyle{definition}
\newtheorem{ex}[thm]{Example}
\theoremstyle{definition}
\newtheorem{prob}[thm]{Problem}
\theoremstyle{definition}
\newtheorem*{question}{Question}

\newcommand{\Maxn}{\operatorname{Max_{\textbf{N}}}}
\newcommand{\Syl}{\operatorname{Syl}}
\newcommand{\dl}{\operatorname{dl}}
\newcommand{\Con}{\operatorname{Con}}
\newcommand{\cl}{\operatorname{cl}}
\newcommand{\Stab}{\operatorname{Stab}}
\newcommand{\Aut}{\operatorname{Aut}}
\newcommand{\Ker}{\operatorname{Ker}}
\newcommand{\fl}{\operatorname{fl}}
\newcommand{\Irr}{\operatorname{Irr}}
\newcommand{\SL}{\operatorname{SL}}
\newcommand{\FF}{\mathbb{F}}
\newcommand{\NN}{\mathbb{N}}
\newcommand{\N}{\mathbf{N}}
\newcommand{\C}{\mathbf{C}}
\newcommand{\OO}{\mathbf{O}}
\newcommand{\bG}{\mathbf{G}}\newcommand{\bK}{\mathbf{K}}
\newcommand{\F}{\mathbf{F}}
\newcommand{\wt}{\widetilde}
\newcommand\wh[1]{\hstretch{2}{\hat{\hstretch{.5}{#1}}}}

\renewcommand{\labelenumi}{\upshape (\roman{enumi})}

\newcommand{\PSL}{\operatorname{PSL}}
\newcommand{\PSU}{\operatorname{PSU}}

\providecommand{\V}{\mathrm{V}}
\providecommand{\E}{\mathrm{E}}
\providecommand{\ir}{\mathrm{Irr_{rv}}}
\providecommand{\Irrr}{\mathrm{Irr_{rv}}}
\providecommand{\re}{\mathrm{Re}}

\def\irr#1{{\rm Irr}(#1)}
\def\irrv#1{{\rm Irr}_{\rm rv}(#1)}
\def \c#1{{\mathcal #1}}
\def\cent#1#2{{\bf C}_{#1}(#2)}
\def\syl#1#2{{\rm Syl}_#1(#2)}
\def\nor{\triangleleft\,}
\def\oh#1#2{{\bf O}_{#1}(#2)}
\def\Oh#1#2{{\bf O}^{#1}(#2)}
\def\zent#1{{\bf Z}(#1)}
\def\det#1{{\rm det}(#1)}
\def\ker#1{{\rm ker}(#1)}
\def\norm#1#2{{\bf N}_{#1}(#2)}
\def\alt#1{{\rm Alt}(#1)}
\def\iitem#1{\goodbreak\par\noindent{\bf #1}}
   \def \mod#1{\, {\rm mod} \, #1 \, }
\def\sbs{\subseteq}

\def\gc{{\bf GC}}
\def\ngc{{non-{\bf GC}}}
\def\ngcs{{non-{\bf GC}$^*$}}
\newcommand{\notd}{{\!\not{|}}}
\def\aut#1{{\rm Aut}(#1)}
\def\inn#1{{\rm Inn}(#1)}
\def\out#1{{\rm Out}(#1)}

\newcommand{\Mult}{{\mathrm {Mult}}}
\newcommand{\Inn}{{\mathrm {Inn}}}
\newcommand{\IBR}{{\mathrm {IBr}}}
\newcommand{\IBRL}{{\mathrm {IBr}}_{\ell}}
\newcommand{\IBRP}{{\mathrm {IBr}}_{p}}
\newcommand{\ord}{{\mathrm {ord}}}
\def\id{\mathop{\mathrm{ id}}\nolimits}
\renewcommand{\Im}{{\mathrm {Im}}}
\newcommand{\Ind}{{\mathrm {Ind}}}
\newcommand{\diag}{{\mathrm {diag}}}
\newcommand{\soc}{{\mathrm {soc}}}
\newcommand{\End}{{\mathrm {End}}}
\newcommand{\sol}{{\mathrm {sol}}}
\newcommand{\Hom}{{\mathrm {Hom}}}
\newcommand{\Mor}{{\mathrm {Mor}}}
\newcommand{\Mat}{{\mathrm {Mat}}}
\def\rank{\mathop{\mathrm{ rank}}\nolimits}
\newcommand{\Tr}{{\mathrm {Tr}}}
\newcommand{\tr}{{\mathrm {tr}}}
\newcommand{\Gal}{{\it Gal}}
\newcommand{\Spec}{{\mathrm {Spec}}}
\newcommand{\ad}{{\mathrm {ad}}}
\newcommand{\Sym}{{\mathrm {Sym}}}
\newcommand{\Char}{{\mathrm {char}}}
\newcommand{\pr}{{\mathrm {pr}}}
\newcommand{\rad}{{\mathrm {rad}}}
\newcommand{\abel}{{\mathrm {abel}}}
\newcommand{\codim}{{\mathrm {codim}}}
\newcommand{\ind}{{\mathrm {ind}}}
\newcommand{\Res}{{\mathrm {Res}}}
\newcommand{\Ann}{{\mathrm {Ann}}}
\newcommand{\Ext}{{\mathrm {Ext}}}
\newcommand{\Alt}{{\mathrm {Alt}}}
\newcommand{\AAA}{{\sf A}}
\newcommand{\SSS}{{\sf S}}
\newcommand{\h}{{\mathcal H}}
\newcommand{\CC}{{\mathbb C}}
\newcommand{\CB}{{\mathbf C}}
\newcommand{\RR}{{\mathbb R}}
\newcommand{\QQ}{{\mathbb Q}}
\newcommand{\ZZ}{{\mathbb Z}}
\newcommand{\NB}{{\mathbf N}}
\newcommand{\OB}{{\mathbf O}}
\newcommand{\ZB}{{\mathbf Z}}
\newcommand{\EE}{{\mathbb E}}
\newcommand{\PP}{{\mathbb P}}
\newcommand{\GC}{{\mathcal G}}
\newcommand{\HC}{{\mathcal H}}
\newcommand{\GA}{{\mathfrak G}}
\newcommand{\TC}{{\mathcal T}}
\newcommand{\SC}{{\mathcal S}}
\newcommand{\RC}{{\mathcal R}}
\newcommand{\GCD}{\GC^{*}}
\newcommand{\TCD}{\TC^{*}}
\newcommand{\FD}{F^{*}}
\newcommand{\GD}{G^{*}}
\newcommand{\HD}{H^{*}}
\newcommand{\GCF}{\GC^{F}}
\newcommand{\TCF}{\TC^{F}}
\newcommand{\PCF}{\PC^{F}}
\newcommand{\GCDF}{(\GC^{*})^{F^{*}}}
\newcommand{\RGTT}{R^{\GC}_{\TC}(\theta)}
\newcommand{\RGTA}{R^{\GC}_{\TC}(1)}
\newcommand{\Om}{\Omega}
\newcommand{\eps}{\epsilon}
\newcommand{\al}{\alpha}
\newcommand{\chis}{\chi_{s}}
\newcommand{\sigmad}{\sigma^{*}}
\newcommand{\PA}{\boldsymbol{\alpha}}
\newcommand{\gam}{\gamma}
\newcommand{\lam}{\lambda}
\newcommand{\la}{\langle}
\newcommand{\ra}{\rangle}
\newcommand{\hs}{\widehat{s}}
\newcommand{\htt}{\widehat{t}}
\newcommand{\tn}{\hspace{0.5mm}^{t}\hspace*{-0.2mm}}
\newcommand{\ta}{\hspace{0.5mm}^{2}\hspace*{-0.2mm}}
\newcommand{\tb}{\hspace{0.5mm}^{3}\hspace*{-0.2mm}}
\def\skipa{\vspace{-1.5mm} & \vspace{-1.5mm} & \vspace{-1.5mm}\\}
\newcommand{\tw}[1]{{}^#1\!}
\renewcommand{\mod}{\bmod \,}

\marginparsep-0.5cm

\renewcommand{\thefootnote}{\fnsymbol{footnote}}
\footnotesep6.5pt

\def\C{{\Bbb C}}
\def\irr#1{{\rm Irr}(#1)}
\def\ibr#1{{\rm IBr}(#1)}
\def\irrp#1#2{{\rm Irr}_{#1'}(#2)}
\def\irrz#1{{\rm Irr}_{0}(#1)}
\def\cent#1#2{{\bf C}_{#1}(#2)}
\def\syl#1#2{{\rm Syl}_#1(#2)}
\def\nor{\triangleleft\,}
\def\zent#1{{\bf Z}(#1)}
\def\norm#1#2{{\bf N}_{#1}(#2)}
\def\oh#1#2{{\bf O}_{#1}(#2)}
\def\iitem#1{\goodbreak\par\noindent{\bf #1}}

\newtheorem*{thmA}{Theorem A}
\newtheorem*{thmB}{Theorem B}
\newtheorem*{thmC}{Theorem C}
\newtheorem*{thmD}{Theorem D}
\newtheorem*{conA}{Conjecture A}
\newtheorem*{conB}{Conjecture B}

\newtheorem{thml}{Theorem}
\renewcommand*{\thethml}{\Alph{thml}}   
\newtheorem{conl}[thml]{Conjecture}

\title{Galois action on the principal block and cyclic Sylow subgroups}

\author{Noelia Rizo}
\address{Dipartimento di Matematica e Informatica U. Dini, Viale Morgagni 67/a, Firenze, Italy}
\email{noelia.rizocarrion@unifi.it}
\author{A. A. Schaeffer Fry}
\address{Dept. Mathematical and Computer Sciences, MSU Denver, Denver, CO 80217, USA}
\email{aschaef6@msudenver.edu}
\author{Carolina Vallejo}
\address{Departamento de Matem\'aticas,  Facultad de Ciencias,
Universidad Aut\'onoma de Madrid, Campus de Cantoblanco, 28049 Madrid, Spain}
\email{carolina.vallejo@uam.es}
\thanks{This material is based upon work supported by the National Security Agency under Grant No. H98230-19-1-0119, The Lyda Hill Foundation, The McGovern Foundation, and Microsoft Research, while the authors were in residence at the Mathematical Sciences Research Institute in Berkeley, California, during the summer of 2019. The first-named author acknowledges support by  Ministerio de Ciencia
e Innovaci\'on PID2019-103854GB-I00 and FEDER funds.
 The second-named author acknowledges support from the National Science Foundation under Grant No. DMS-1801156.
The third-named author acknowledges support by  Ministerio de Ciencia
e Innovaci\'on projects MTM2017-82690-P and PID2019-103854GB-I00, FEDER funds and the ICMAT Severo Ochoa project SEV-2011-0087.}



\begin{abstract}
 We characterize finite groups $G$ having a cyclic Sylow $p$-subgroup in terms of the action of a specific Galois automorphism on the principal $p$-block of $G$, for $p=2, 3$. We show that the analog statement for blocks with arbitrary defect group would follow from the blockwise McKay--Navarro conjecture.
 \end{abstract}

\maketitle

\section*{Introduction}
\noindent One of the most prevalent questions in the representation theory of finite groups is to determine what relationships hold between the
set $\irr G$ of irreducible complex characters of a finite group $G$ and its local structure, such as the structure of a Sylow $p$-subgroup $P$ of $G$. There 
is, of course, the more sophisticated question of relating the set $\irr B$ of irreducible characters belonging to a given Brauer $p$-block $B$ of $G$
with the structure of a defect  group $D$ of $B$. 

In \cite{NT19}, G. Navarro and P. H. Tiep conjecture that for a prime $p$, one can determine the exponent of the abelianization of $P$
in terms of the action of certain Galois automorphisms on $\irr G$. 
To be more precise, for a fixed prime $p$ and an integer $e\geq 1$, let $\sigma_e\in {\rm Gal}(\QQ^{ab}/\QQ)=\c G$ be such that $\sigma_e$ fixes $p'$-roots of unity and sends any root of unity of order a power of $p$ to its $(p^e+1)$st power.  In \cite{NT19} it is proven that the exponent of $P/P'$ is less than or equal to $p^e$ whenever all of the irreducible characters of $p'$-degree of $G$ are $\sigma_e$-fixed, and the converse is reduced to a question on finite simple groups. (Thanks to \cite{Mal19} we know that the converse holds for $p=2$.) 

In the present work, we show that one can determine whether $P$ is cyclic (for small primes) by just counting the number of certain $\sigma_1$-invariant elements of $\irr {B_0}$, where $B_0$ is the principal $p$-block of $G$. 
This is the main result of our paper.

\begin{thml}\label{thmA}
{\em Let $G$ be a finite group of order divisible by $p$, where $p\in \{ 2, 3\}$. Let $P\in \syl p G$ and let $B_0$ be the principal $p$-block of $G$. Then 
$$|\irrp p  {B_0}^{\sigma_1}|=p \text{\  if, and only if, } P \text {  is cyclic,}$$
where $\irrp p  {B_0}^{\sigma_1}$ is the set of irreducible characters in $B_0$ with degree relatively prime to $p$ that are fixed under the action 
of $\sigma_1$.}
\end{thml}


With the definition above, $\sigma_1$ is an element of the subgroup $\c H\leq \c G$ consisting of all $\sigma \in \c G$ for which there exists some integer $f$ such that $\sigma(\xi)=\xi^{p ^f}$ whenever $\xi$ is a root of unity of order not divisible by $p$. Navarro predicted in \cite[Conjecture A]{Nav04} the existence of bijections for the McKay conjecture commuting with the action of $\c H$ on characters. This is the celebrated McKay--Navarro conjecture (sometimes also referred to as the Galois--McKay conjecture), 
which has been recently reduced to a question on finite simple groups in \cite{NSV19}. The McKay--Navarro conjecture admits a blockwise version \cite[Conjecture B]{Nav04}, which remains unreduced at the present moment and which we will refer to as the Alperin--McKay--Navarro conjecture, as it can also be seen as a refined version of the celebrated Alperin--McKay conjecture. In this context, it is natural to wonder the extent to which Theorem \ref{thmA} holds for arbitrary blocks. We propose the following.

\begin{conl}\label{conB}
{\em Let $p \in \{ 2, 3\}$. Let $G$ be a finite group and let $B$ be a $p$-block of $G$ with nontrivial defect group $D$. Then
 $$|\irrz  {B}^{\sigma_1}|=p \text{   if, and only if, } D \text { is cyclic,}$$
 where $\irrz B^{\sigma_1}$ is the set of height zero irreducible characters in the block $B$ that are fixed under the action of $\sigma_1$.}
\end{conl}

We prove that Conjecture \ref{conB} follows from the Alperin--McKay--Navarro conjecture. In this sense, Theorem \ref{thmA} provides more evidence of the elusive Alperin--McKay--Navarro conjecture.   Since the latter holds whenever $D$ is cyclic, by work of Navarro in \cite{Nav04}, it follows that the ``if'' direction of Conjecture \ref{conB} (and of Theorem \ref{thmA}) holds.
For many consequences of the (Alperin-)McKay-Navarro conjecture, the statements take different forms depending on the prime (see, for instance, \cite{NTT07} and \cite{SF19}). This might well be the case here, however, we are not yet aware of such a statement for $p>3$.

To prove Theorem \ref{thmA}, we use the Classification of Finite Simple Groups.  
In particular, we contribute to the problem of understanding Galois action on the characters in blocks of nonabelian simple groups
in the following way. 

\begin{thml}\label{Mandi3}\label{simplegroupsp2}\label{simplegroups2}
Let $S$ be a nonabelian simple group of order divisible by $p\leq 3$, $P \in \syl p S$ and $X \in \syl p {\aut S}$. Let
$B_0$ be the principal $p$-block of $S$. 
 
\begin{enumerate}

\item[{\rm (a)}] If $P$ is cyclic, then $p=3$ and $\irrp p {B_0}^{\sigma_1}=\{ 1_S, \phi_1, \phi_{2}\}$, where the $\phi_i$ are nontrivial and not $\aut S$-conjugate, 
and some $\phi_i$ is $X$-invariant.

\item[{\rm (b)}]  If $P$ is not cyclic, then $\irrp p {B_0}^{\sigma_1}\supseteq \{ 1_S, \phi_1, \ldots, \phi_p \}$, where the nontrivial $\phi_i$
are pairwise not $\aut S$-conjugate, and some $\phi_i$ is $X$-invariant.
\end{enumerate}
\end{thml}

This paper is structured as follows. In Section \ref{normaldefect} we prove that Conjecture \ref{conB} follows from the Alperin--McKay--Navarro conjecture. To do so, we study the action of $\sigma_1$ on the irreducible characters of blocks with normal defect group. The rest of the paper is devoted to proving Theorem \ref{thmA}. In Section \ref{reductions}, we reduce Theorem \ref{thmA} to statements on finite simple groups, and in Section \ref{simplegroups} we prove Theorem \ref{Mandi3} thus completing the proof of Theorem \ref{thmA}.

\medskip

\noindent {\bf Acknowledgements.}
The authors would like to thank Gabriel Navarro for an inspiring conversation on Galois action and generating properties of Sylow subgroups during the Workshop on Representations of Finite Groups at the 
MFO in March 2019, and for further discussion on the topic.
They would also like to thank the 
MSRI in Berkeley, CA, and its generous staff for providing a collaborative and productive work environment during their residency in the summer of 2019.  In addition, they thank Gunter Malle for his thorough reading of an early draft of this manuscript and for his helpful comments.
Part of this work was done while the first-named author was visiting the Department of Mathematical and Computer Sciences at MSU Denver and the Department of Mathematics at the TU Kaiserslautern. She would like to thank everyone at both departments for their warm hospitality.
Last but not least, the authors are indebted to the anonymous referee for valuable suggestions.   


\section{Blocks with normal defect group}\label{normaldefect}

The aim of this section is to prove that Conjecture \ref{conB} follows from the Alperin--McKay--Navarro conjecture, stated below.

For a fixed prime $p$, consider the set ${\rm Bl}(G)$ of Brauer ($p$-)blocks of $G$ as in \cite{Nav98}, so that ${\rm Bl}(G)$ is a partition of $\irr{G}\cup \ibr G$ (recall that $p$-Brauer characters are defined on $p$-regular elements of $G$). Write $\irr B=B\cap \irr G$ and $\ibr B =B \cap\ibr G$ for any $B \in {\rm Bl}(G)$. Every block $B$ has associated a uniquely defined conjugacy class of $p$-subgroups of $G$, namely its defect groups. Given a block $B$ of $G$ with defect group $D$, we write $B \in {\rm Bl}(G|D)$ and we let $b \in {\rm Bl}(\norm G D | D)$ denote its Brauer first main correspondent. Finally,  $\chi\in{\rm Irr}(B)$ has height zero in $B$ if $\chi(1)_p=|G:D|_p$, and we write $\irrz B$ to denote the subset of height zero characters in $\irr B$.

Assuming the notation of the Introduction, we have that the group $\c G$ acts on $\{ \irr B | \ B\in {\rm Bl}(G)\}$ by \cite[Theorem 3.19]{Nav98}. The group $\c H$ further acts on the set ${\rm Bl}(G)$ by \cite[Theorem 2.1]{Nav04}. While the action of $\c G$ on characters is not natural enough in global-local contexts, Navarro conjectured the following in \cite{Nav04}.
\begin{conj}[Alperin--McKay--Navarro conjecture]
Let $B \in {\rm Bl}(G|D)$ and let $b \in {\rm Bl}(\norm G D | D)$ be its Brauer first main correspondent. If $\sigma \in \c H$, then 
$$|\irrz B^{\sigma}|=|\irrz b ^{\sigma}|\, .$$
\end{conj}
Here we are only concerned with the action of a specific element of $\c H$, namely $\sigma_1$. Recall that  $\sigma_1 \in \c H$ fixes $p'$-roots of unity and sends any root of unity of order a power of $p$ to its $(p+1)$st power. If $G$ is a finite group of order dividing some integer $n$ and $\xi_n$ is a primitive $n$th root of unity, then by elementary number
theory, the restriction $\omega$ of $\sigma_1$ to the $n$th cyclotomic field $\QQ(\xi_n)$ has order
a power of $p$, and $\omega$ acts as $\sigma_1$ on the ordinary characters of every subgroup of $G$. 
Abusing notation, we will also write $\sigma_1$ for any  such restriction. In particular, $\sigma_1$ fixes the elements of $\ibr G$, and hence acts trivially on ${\rm Bl }(G)$. (Note that in general $\mathcal G$ does not act on $\ibr G$, but $\mathcal H$ does by Theorem 2.1 of \cite{Nav04}.)


In order to prove that Conjecture \ref{conB} follows from the Alperin--McKay--Navarro conjecture, we need to study blocks with a normal defect group. We follow the notation in Chaper 9 of \cite{Nav98}. Let $B \in {\rm Bl}(G|D)$ and assume that $D\nor G$.  Write $C=\cent G D$. We will denote by $b \in {\rm Bl}(C D| D)$ a {\em root} of $B$, and we will let $\theta \in \irr b$ be the {\em canonical character} associated with $B$, which is unique up to $G$-conjugacy (see \cite[Theorem 9.12]{Nav98} and the subsequent discussion). Recall that $D\sbs \ker \theta$ and $\theta$ has $p$-defect zero when viewed as a character of $CD/D$ (that is, $\theta(1)_p=|CD:D|_p$), the stabilizer of the block $b$ is $G_b=G_\theta$, and the inertial index $|G_\theta:CD|$ is not divisible by $p$.
In this situation, $\irr b=\{ \theta_\lambda \ | \ \lambda \in \irr D\}$, where the irreducible characters $\theta_\lambda\in\irr{CD}$ are defined for $x\in CD$ as follows:  $\theta_\lambda(x)=\lambda(x_p)\theta(x_{p'})$ if $x_p\in D$ and $\theta_\lambda(x)=0$ otherwise. One can see that \[G_{\theta_\lambda}=G_\theta\cap G_\lambda.\] 
Let $c \in {\rm Bl}(G_b|D)$ be the  Fong-Reynolds correspondent of $b$ and $B$ as in \cite[Theorem 9.14]{Nav98}. Then the induction map $\irr c\to \irr B$ defines a height-preserving bijection. By \cite[Theorems 9.21 and 9.22]{Nav98} $c=b^{G_b}$ is the only block of $G_b$ that covers $b$ and 
\begin{equation}\label{IrrB} \irr B=\bigcup_{\lambda \in \irr {D}} \irr{G|\theta_\lambda}\, .\end{equation}
It is not difficult to see that height zero characters of $B$ further lie over characters parametrized by linear characters of $D$, so that
\begin{equation}\label{Irr0B}\irrz B=\bigcup_{\lambda \in \irr {D/D'}} \irr{G|\theta_\lambda}\, .\end{equation}

In order to explicitly describe the set $\irrz B ^{\sigma_1}$ when the defect group of $B$ is normal we will use the following technical lemma.

\begin{lem}\label{stabilizers}  
Let $G$ be a finite group and let $p$ be a prime. Suppose that $B$ is a block of $G$ with normal defect group $D$. 
Let $b$ be a root of $B$ with canonical character $\theta$.  Write $A=\langle \sigma_1\rangle\leq {\rm Gal}(\mathbb{Q}(\xi_{|G|})/\mathbb{Q})$. 
If $\lambda$   is a linear character of $D$, then let
$G_{\theta_\lambda^A}=\{ g \in G \ | \ (\theta_\lambda)^g=(\theta_\lambda)^a \text{ \ for some \ } a \in A \}$. With this definition $$G_{\theta_\lambda^A}=G_{\theta_\lambda}=G_\theta \cap G_\lambda\, .$$
\end{lem}

\begin{proof}
Write $C=\cent G D$. Recall that $b$ is a block of $CD$ of defect $D$ and $\theta \in \irr {CD}$ has defect zero as a character of $CD/D$. Note that $\theta$ is $A$-fixed since $b^a=b$ for every $a \in A$. Let $g \in G_{\theta_\lambda^A}$. We start by proving that $g\in G_\theta$. Since $\theta$ is $A$-fixed, by the definition of $\theta_\lambda$ we have $(\theta_\lambda)^g=(\theta_{\lambda})^a=\theta_{\lambda^a}$ for some $a \in A$. 
Evaluating on $D$ we see that
$$\theta(1)\lambda^{a}(x)=\theta_{\lambda^a}(x)=\theta_\lambda^g(x)=\theta(1)\lambda^g(x) \, ,$$
for every $x \in D$. Hence $\lambda ^g =\lambda^a$. Let $x \in CD$ be such that $xD\in (CD/D)^0$, the set of $p$-regular elements of $CD/D$, and notice that $x_p\in D$. (Otherwise $\theta(x)=0$.) Then
$$\lambda^g(x_p)\theta^g(x_{p'})=\theta_\lambda^g(x)=\theta_\lambda^a(x)=\lambda^a(x_p)\theta(x_p')=\lambda^g(x_p)\theta(x_p')\, .$$ This implies $\theta^g(x_{p'})=\theta(x_{p'})$. Since $xD=x_{p'}D$, then $\theta^g =\theta$ and $g \in G_\theta$. 

Next we prove that $g\in G_\lambda$. We know that $\lambda^g =\lambda^a$ for some $a \in A$, and that $g \in G_\theta$.
Since $G_\theta/CD$ is a $p'$-group, then $\lambda^{g^m}=\lambda$ for some integer $m$ relatively prime to $p$. In particular,
$\lambda^{a^m}=\lambda$ and the order of $a$ as an 
element of ${\rm Gal}(\QQ(\lambda)/\QQ)={\rm Gal}(\mathbb{Q}(\xi_{o(\lambda)})/\mathbb{Q})$ divides $m$, which forces $a=1$ and $\lambda^g=\lambda$, as wanted. 
\end{proof}




\begin{lem}\label{generalNavarro}
Let $G$ be a finite group and let $p$ be a prime. Suppose that $B$ is a block of $G$ with a normal defect group $D$. 
Let $b$ be a root of $B$ with canonical character $\theta$. Then
$$\irrz B ^{\sigma_1}=\bigcup_{\lambda \in \irr{D/\Phi(D)}}\irr{G|\theta_\lambda}\, ,$$
where $\Phi(D)$ is the Frattini subgroup of $D$. Moreover, if $c \in {\rm Bl}(G_b|D)$ is the Fong-Reynolds correspondent of
$B$, then
$$|\irrz B ^{\sigma_1}|=|\irrz {c}^{\sigma_1}| \, .$$
\end{lem}
\begin{proof}
First notice that as a $p$-group, $D$ has a unique block, the principal one, and $\irrz {B_0(D)}=\irrp p {D}=\irr{D/D'}$. Then $\irrp p{D}^{\sigma_1}=\irr{D/\Phi(D)}$. Since $D/\Phi(D)$ is 
$p$-elementary abelian, one inclusion is straight-forward. To see that $\irrp p {D}^{\sigma_1}\sbs \irr{D/\Phi(D)}$ notice that if $\lambda\in \irr {D/D'}$ is $\sigma_1$-fixed, then $\lambda^{\sigma_1}=\lambda^{p+1}=\lambda$, and hence $|D/\ker \lambda|\leq p$, implying $ \Phi(D)\sbs \ker \lambda$.

Write $A=\langle \sigma_1 \rangle$ and let $G_{\theta_\lambda^A}$ be as in Lemma \ref{stabilizers}. By Equation (\ref{Irr0B}), we know that
$$\irrz B=\bigcup_{\lambda \in \irr {D/D'}} \irr{G|\theta_\lambda}\, .$$
If $\chi \in \irrz B^{\sigma_1}$ lies over $\theta_\lambda$, then $(\theta_\lambda )^{\sigma_1}=(\theta_\lambda) ^g$, for some $g \in G$. In particular,
$g \in G_{\theta_\lambda^A}=G_{\theta_\lambda}=G_\theta\cap G_\lambda$ by Lemma \ref{stabilizers}. Then $\lambda^{\sigma_1}=\lambda^g =\lambda$. Hence $\Phi(D)\sbs \ker \lambda$ and $\lambda \in \irr{D/\Phi(D)}$.

\smallskip

Conversely, let $\chi \in  \irr{G|\theta_\lambda}$, where $\lambda \in \irr{D/\Phi(D)}$. Then $\lambda^{\sigma_1}=\lambda$. As $b^{\sigma_1}=b$, we see $\sigma_1$ fixes $\theta$ too. Then $(\theta_\lambda)^{\sigma_1}=\theta_\lambda$. Let $\psi \in \irr{G_{\theta_\lambda}}$ be the Clifford correspondent of $\chi$ over $\theta_\lambda$. Since $G_{\theta_\lambda}\sbs G_\theta$, we know that $p$ does not divide the order of $G_{\theta_\lambda}/CD$. By Lemma  \cite[Lemma 5.1]{NT19}, $\psi$ is $\sigma_1$-invariant and so is $\chi$.

\smallskip

To prove the last part of the statement, recall that the Fong-Reynolds correspondence states that the induction map $\psi \mapsto \psi^G$ provides a bijection  $\irrz {c}\rightarrow \irrz B $. In particular, $|\irrz {c}^{\sigma_1}|\leq |\irrz B^{\sigma_1}|$.
Now let $\chi \in \irrz B ^{\sigma_1}$ lie over $\theta_\lambda$, for some $\lambda \in \irr{D/\Phi(D)}$ by the first part of this proof. Then $(\theta_\lambda)^{\sigma_1}=(\theta_\lambda)^g$
 for some $g \in G$. In particular, $g \in G_{\theta_\lambda^A}$. Since $G_{\theta_\lambda^A} =G_{\theta_\lambda}$ by Lemma \ref{stabilizers}, 
$\theta_\lambda$ is $\sigma_1$-fixed. Let $\xi \in \irr{G_{\theta_\lambda}|\theta_\lambda}$ be the Clifford correspondent of $\chi$. Since both $\chi$ and $\theta_\lambda$ are $\sigma_1$-fixed then so is $\xi$. We have that $\xi^{G_b}$ is the Fong-Reynolds correspondent of $\chi$ by the transitivity of block induction (see \cite[Problem 4.2]{Nav98}), which is $\sigma_1$-fixed. 
\end{proof}

 The Alperin--McKay--Navarro conjecture holds for blocks with cyclic defect groups by \cite[Theorem 3.4]{Nav04}. We obtain the following as a consequence of
this fact. 

\begin{lem}\label{thmbound}
Let $G$ be a finite group and let $B$ be a block of $G$ with cyclic defect group $D$. Then 
$$1\leq |\irrz B^{\sigma_1}|\leq p\, .$$
The set $\irrz B^{\sigma_1}$ has minimal size 1 if, and only if, $D$ is trivial. 
Furthermore, if $p\in \{2,3\}$ and $D$ is nontrivial, then 
$$|\irrz B ^{\sigma_1}|=p\, .$$
\end{lem}

\begin{proof}
By \cite[Theorem 3.4]{Nav04}, we may assume that $D \nor G$. Write $C=\cent G D\supseteq D$. Let $b \in {\rm Bl}(C| D)$ be a root of $B$ with canonical character $\theta$.
By Lemma \ref{generalNavarro}, we may assume that $\theta$ is $G$-invariant (in particular, $G/C$ is a $p'$-group) and
$$\irrz B^{\sigma_1}=\bigcup_{\lambda \in \irr{D/\Phi(D)}}\irr{G|\theta_\lambda}\, \sbs{\rm Irr}(G/\Phi(D)).$$
Write $\overline{G}=G/\Phi(D)$ and use the bar convention. Let $\overline F=\cent {\overline G} {\overline D}$, where $\Phi(D)\sbs F\leq G$. We
claim that $F=C$. Clearly $C\sbs F$. Note that $\overline F$ acts trivially on $\overline D$ and coprimely on $D$. By  \cite[Theorem 3.29]{Isa08} we have that $\overline F$
acts trivially on $D$ as well. Thus $F=C$ as claimed.

 
 Notice that since $D$ is cyclic and $G/C$ is a $p'$-group, then $G/C$ is isomorphic to a subgroup of ${\sf C}_{p-1}$. Say $|G/C|=m$ and 
 let $\{ \lambda_i \}_{i=1}^t$ be a complete set of representatives of the $G/C$-orbits on $\irr {\overline{D}}\setminus \{ 1_D\}$, where
here  we view $\irr {\overline D}\sbs \irr D$, and with this identification $\irr {\overline D}$ are exactly the elements of $\irr D$ with order dividing $p$.
  Note that $\ker {\lambda_i}=\Phi(D)$ for all $1\leq i\leq t$, hence $G_{\lambda_i}=C$  for every $1\leq i\leq t$, and all the orbits of the action of $G/C$ on ${\rm Irr}(\overline{D})\setminus\{1_D\}$ have the same size $m$. In particular, $t=\frac{p-1} m$. 
Since $\theta$ is $G$-invariant, for every $1\leq i \leq t$ we have that $G_{\theta_{\lambda_i}}=G_{\lambda_i}=C$, and by the Clifford correspondence, $|{\rm Irr}(G|\theta_{\lambda_i})|=|{\rm Irr}(C|\theta_{\lambda_i})|=1$. Also, since $G/C$ is cyclic, $\theta$ extends to $G$ and therefore by Gallagher theory $|{\rm Irr}(G|\theta)|=m$. Then $$|\irrz B^{\sigma_1}|= |\irr{G|\theta}|+\sum_{i=1}^t |\irr{G|\theta_{\lambda_i}}|=m+t=m+\frac{p-1}m\leq p\,.$$
Note that if $p=2,3$ then $m+\frac{p-1} m=p$, whenever $m$ divides $p-1$. Also notice that
$|\irrz B^{\sigma_1}|=1$ if, and only if, $D=1$.
\end{proof}

The upper bound in Lemma \ref{thmbound} is not generally attained if $p>3$, as shown by the dihedral group ${\sf D}_{2p}$, which
satisfies $|\irr{B_0({\sf D}_{2p}})^{\sigma_1}|<p$. We care to remark that the numerical condition $|\irrz {B}^{\sigma_1}|\leq p$ does not generally imply that a defect group $D$ of $B$ is cyclic. For instance, for $p=11$, the semidirect product $H=\mathbb{F}_{11}^2\rtimes{ SL}_2(5)$ satisfies $|{\rm Irr}_{11'}(B_0(H))^{\sigma_1}|=|\irr H|=10$. (We would like to thank Gabriel Navarro for providing us with this example.) 

\smallskip

We will need the following divisibility result, which we obtain by adaptating the proof of \cite[Theorem 5.2]{Gow79}. 

\begin{lem}\label{divisible_by_p} Let $G$ be a finite group, let $p\in\{2,3\}$, and let $B$ be a block of $G$ with nontrivial defect group $D$. Then $p$ divides $|{\rm Irr}_0(B)^{\sigma_1}|$. 
\end{lem}
\begin{proof}
Write 
\begin{equation}\label{divisibility}\psi=\sum_{\chi \in \irr{B}} \chi(1) \chi\, ,\end{equation} and notice that $\psi$
is a character of $G$ that vanishes on $p$-singular elements by the weak block orthogonality relation (see \cite[Corollary 3.7]{Nav98}). In particular,
$\psi_P=f \rho_P$ for some natural number $f$, where $\rho_P$ denotes the regular character of $P$.

Let $\irr  {B}=\{ \chi_1, \ldots, \chi_t\}$ and write $\chi_i(1)=p^{a-d+h_i}b_i$, where $|P|=p^a$, $|D|=p^d$, $h_i\geq 0$ is the height of $\chi_i$ and $p$ does not divide $b_i$, for $1\leq i\leq t$. Arrange the elements in $\irr B$ in such a way that $\irrz  {B}=\{ \chi_1, \ldots, \chi_k\}$, so that $h_j\geq 1$ for all $k+1\leq j\leq t$. By  \cite[Theorem 3.28]{Nav98} we have that $\psi(1)=p^{2a-d}c$, where $c$ is a non-negative integer relatively prime to $p$. Thus, evaluating (\ref{divisibility}) at $1 \in G$ we obtain
$$p^dc=\sum_{i=1}^k b_i^2 + \sum_{j=k+1}^t p^{2h_j}b_j^2 \, .$$
As $d\geq 1$, we get $\sum_{i=1}^k b_i^2\equiv 0$ mod $p$. Since $p\in\{2,3\}$, we have that $b_i^2\equiv 1$ mod $p$ for every $1\leq i \leq k$, and hence $k$ is divisible by $p$.


Recall that the group $A=\langle \sigma_1 \rangle$ acts on $\irrz {B}$, and as such, we may view $A$ as having order a power of $p$. Since
$|\irrz {B}^{\sigma_1}|=|\irrz {B}^A|$, we obtain that $p$ divides $|\irrz {B}^{\sigma_1}|$ by the class equation for group actions.
\end{proof}

The conclusion of the result above does not hold if $p>3$, as the dihedral group ${\sf D}_{2p}$ provides a counterexample. Indeed, ${\sf D}_{2p}$  has a unique $p$-block and every irreducible character has $p'$-degree and is $\sigma_1$-fixed. Hence $|\irrp p {B_0({\sf D}_{2p})}^{\sigma_1}|=|\irr{{\sf D}_{2p}}|=2+(p-1)/2<p$. 

\medskip

Finally, we prove the main result of this section.

\begin{thm}\label{thmnormaldefect}
Let $p \in \{ 2, 3\}$. Let $G$ be a finite group and let $B$ be a $p$-block of $G$ with a nontrivial normal defect group $D$. Then
 $$|\irrz  {B}^{\sigma_1}|=p \text{  if, and only if, } D \text { is cyclic.}$$
 In particular, Conjecture \ref{conB} follows from the Alperin--McKay--Navarro conjecture. 
\end{thm}
 
\begin{proof}
By Lemma \ref{thmbound} we know that the ``if" implication holds. We now assume that
$|\irrz B^{\sigma_1}|=p$ and we work to show that $D$ is cyclic. 

\smallskip

 Write $C=\cent G D$ and let $\theta \in \irr{CD}$ be the canonical character of $B$. Let $\{ \lambda_i\}_{i=1}^t$ be a complete set of representatives of the $G/CD$-orbits on $\irr {D/\Phi(D)}\setminus \{ 1_D\}$. By Lemma \ref{generalNavarro} we may assume that $G_\theta=G$ and $$\irrz B^{\sigma_1}=\bigcup_{i=1}^t\irr{G|\theta_{\lambda_i}} \, $$
 is a disjoint union. If $p=2$, then $$2=|\irrz B^{\sigma_1}|=|\irr{G|\theta}|+\sum_{i=1}^t |\irr{G|\theta_{\lambda_i}}| \,.$$
Since $D$ is nontrivial by hypothesis, we have that $t \geq 1$. Thus $t=1$ and the characters
$\theta$ and $\theta_{\lambda_1}$ are fully ramified with respect to their inertia subgroups. In particular, there are positive integers $e$ and $e_1$ such that
$|G:C|=e^2$ and $|G_{\theta_{\lambda_1}}:C|=e_1^2$. Suppose that $|D|=2^n$. Since $G/CD$ acts transitively on the nontrivial elements of $D/\Phi(D)$, we have that
$2^n-1=|G:G_{\lambda_1}|=(\frac{e}{e_1})^2=f^2\, .$
The equality $f^2+1=2^n$ forces $f$ to be odd, then $f^2\equiv 1 \mod 8$, and so $f^2+1\equiv 2 \mod 8$ leaves as the only possibility $n=1=f$, that is, $D={\sf C}_2$, as wanted. 
These techniques do not totally suffice to prove the case where $p=3$.
We first need to show that we may assume $\Phi(D)=1$. Indeed, write $\overline{G}=G/\Phi(D)$, $\overline{D}=D/\Phi(D)$ and let $\overline{B}$ be a block of $\overline{G}$ contained in $B$ such that $\overline{D}$ is the defect group of $\overline{B}$ by \cite[Theorem 9.9]{Nav98}. Then  ${\rm Irr}_0(\overline{B})^{\sigma_1}\subseteq{\rm Irr}_0(B)^{\sigma_1}$. By Lemma \ref{generalNavarro} we have that ${\rm Irr}_0(\overline{B})^{\sigma_1}={\rm Irr}(\overline{B})$ is non-empty.  Hence by Lemma \ref{divisible_by_p}, we have that $p$ divides $|{\rm Irr}_0(\overline{B})^{\sigma_1}|\leq|{\rm Irr}_0(B)^{\sigma_1}|=p$, that forces
 $|{\rm Irr}_0(\overline{B})^{\sigma_1}|=p$. If $\Phi(D)\neq 1$ we can apply induction to obtain that $\overline{D}$ is cyclic, and thus $D$ is cyclic. Hence we may assume that $\Phi(D)=1$.
Since $D$ is $p$-elementary abelian, then
$\irrz B^{\sigma_1}=\irr B$ by  the description of these sets in Equation (\ref{IrrB}) and Lemma \ref{generalNavarro}. By \cite[Proposition 15.2]{Sam14}, if $p=3$ then  $|{\rm Irr}(B)|=p$ implies $|D|=p$ and the proof is finished.
 \end{proof}


\section{Reducing to simple groups}\label{reductions}

The aim of this section is to reduce the statement of Theorem \ref{thmA} to a problem on simple groups that we will solve in Section \ref{simplegroups}.

\subsection{Preliminaries}
We start these preliminaries with results concerning the action of Galois automorphisms on characters belonging to principal blocks. 
Recall that $\chi \in \irr{B_0(G)}$ if, and only if, $$\sum_{x \in G^0} \chi(x)\neq 0 \, ,$$where $G^0$ is the subset of elements of $G$ of order not divisible by $p$. 
Some properties of characters in the principal block are listed below. 
\begin{lem}\label{basicsprincipalblock} Let $G$ be a finite group, and let $N\nor G$.
\begin{enumerate}[{\rm (a)}] 
\item We have that $\irr{B_0(G/N)}\sbs \irr{B_0(G)}$, with equality whenever $N$ is a $p'$-group.
\item If $H_i$ are finite groups and $\gamma_i\in \irr{B_0(H_i)}$, for $i=1, \ldots, t$, then $\gamma_1\times \cdots \times \gamma_t\in \irr{B_0(H_1\times \cdots \times H_t)}$.
\end{enumerate}
\end{lem}
\begin{proof} The first part of (a) and (b) follow directly from the definition of principal block \cite[Definition 3.1]{Nav98}. The second part of (a)
is \cite[Theorem 9.9.(c)]{Nav98}.
\end{proof}

We summarize below some results obtained in Section 1, here stated with respect to the principal block. The first part was first observed by G. Navarro (in private communication). 

\begin{lem}\label{Navarroandboundforprincipalblock}
Let $G$ be a finite group and let $P$ be a Sylow $p$-subgroup of $G$. 

\begin{enumerate}[{\rm (a)}]
\item If $P$ is normal in $G$, then $\irrp p {B_0(G)}^{\sigma_1}=\irr{G/\oh{p' } G\Phi(P)}\, .$
\item If $P$ is cyclic, then $1\leq |\irrp p {B_0(G)}^{\sigma_1}|\leq p\, .$
\item  If $P$ is nontrivial and $p \in \{ 2, 3 \}$, then $|\irrp p {B_0(G)}^{\sigma_1}|\neq 0$ is divisible by $p$.
\end{enumerate}
\end{lem}

\begin{proof}
To prove part (a), assume that $P \nor G$.  Then $G$ is $p$-solvable and by Fong's theorem \cite[Theorem 10.20]{Nav98} 
$\irr{B_0(G)}=\irr{G/\oh{p'} G}$. Hence we may assume that $\oh{p'}G =1$ and, in particular, $\cent G P \sbs P$. 
By Lemma \ref{generalNavarro}
$$\irrp p {B_0(G)}^{\sigma_1}=\irrp p {G}^{\sigma_1}=\bigcup _{\lambda \in \irr{P/\Phi(P)}}\irr{G|\lambda}=\irr{G/\Phi(P)} \, .$$
Part (b) is a straightforward application of Lemma \ref{thmbound}. Part (c) is a direct consequence of Lemma \ref{divisible_by_p}.
\end{proof}

Next is a classical result by J. L. Alperin and E. C. Dade.

\begin{thm}\label{isomblocks}
Suppose that $N$ is a normal subgroup of $G$ and $G/N$ is a $p'$-group.
Let $P \in \syl p G$ and assume that $G=N\cent GP$. Then restriction of characters defines
a bijection $\irr{B_0(G)}\to \irr{B_0(N)}$. In particular, $|{\rm Irr}_{p'}(B_0(G))^{\sigma_1}|=|{\rm Irr}_{p'}(B_0(N))^{\sigma_1}|$.
\end{thm}
\begin{proof}
The case where $G/N$ is solvable was proved in \cite[Lemma 1.1]{Alp76}. The general case
in the main result of \cite{Dad77}. The latter statement follows since $\sigma_1$ acts on $\irrp p {B_0(G)}$.
\end{proof}

We will also use the following.
 
 \begin{lem}\label{inducing}
 Suppose that $G$ is a finite group, $P \in {\rm Syl}_p(G)$ and $P\cent GP \le H \le G$.
If $\theta \in \irrp p {B_0(H)}^{\sigma_1}$, then there exists a some $\chi \in \irrp p {B_0(G)}^{\sigma_1}$ 
lying over $\theta$.
 \end{lem}
 \begin{proof}
 Note that $B_0(H)^G=B_0(G)$ by the comments before \cite[Theorem 9.24]{Nav98} and Brauer's third main theorem \cite[Theorem 6.7]{Nav98}. Write
 $$\Psi=\sum_{\chi \in \irr {B_0(G)}}[\theta^G, \chi]\chi,$$
 so that $\Psi$ has $p'$-degree by \cite[Theorem 6.4]{Nav98}. (Note that $\Psi$ is exactly $(\theta^G)_B$ where $B=B_0(G)$ in \cite{Nav98}'s notation.)
 Let $A=\langle \sigma_1\rangle$, where here we view $\sigma_1$ as an element of ${\rm Gal}(\QQ(\xi)/\QQ)$ for $o(\xi)=|G|$. For every $a \in A$ we have that
 $\Psi^a =\Psi$ as $A$ acts on $\irr{B_0(G)}$ and fixes $\theta$. By \cite[Lemma 2.1.(ii)]{NT19} there is some $\chi \in \irrp p {G}^{\sigma_1}$ appearing with $p'$-multiplicity in $\Psi$. The statement now follows since every irreducible constituent of $\Psi$ lies in the principal block and its multiplicity in $\Psi$ is exactly the multiplicity of $\theta$ in its restriction to $H$. 
 \end{proof}

We end the preliminaries with a technical result. 

\begin{lem}\label{invariance_direct_products}
Let $G$ be a finite group and let $N\nor G$ be a direct product of $t$ copies of a simple nonabelian group $S$ transitively permuted by $G$. Let $P\in \syl p G$.
If some $1_S \neq \phi \in \irrp p {B_0(S)}^{\sigma_1}$ is $X$-invariant, where $X \in \syl p {\aut S}$, then there exists some $P$-invariant $1_N\neq \theta \in \irrp p {B_0(N)}^{\sigma_1}$. In particular, if $N$ is a minimal normal subgroup of $G$, then $\theta$ extends to a $\sigma_1$-invariant irreducible character of $PN$. 
\end{lem}

\begin{proof}
Let  $1_N\neq \theta \in \irr N$ the character of $N$ corresponding to $\phi\times \cdots \times \phi \in \irrp p {B_0(S)^t}^{\sigma_1}$, then
$\theta \in \irrp p {B_0(N)}^{\sigma_1}$ by Lemma \ref{basicsprincipalblock}(b). By  \cite[Lemma 4.1.(ii)]{NTT07}, we may assume that $\theta$ is $P$-invariant. 

For the second part of the statement, notice that since $PN/N$ is a $p$-group and $N$ is perfect, $\theta$ has a canonical extension $\hat{\theta}\in{\rm Irr}_{p'}(PN)$ by \cite[Corollary 6.28]{Is}. In particular, $\hat \theta$ is $\sigma_1$-invariant.
\end{proof}

\subsection{The reduction}

Here we reduce Theorem \ref{thmA} to a problem on simple groups, which is done in Theorem \ref{pequals3} below. Theorem \ref{Mandi3} collects the properties of simple groups that will be key for performing such reduction. We would like to remark that the conditions in Theorem \ref{Mandi3} related to the conjugation by group automorphisms are not needed in this context, but may be of independent interest. 

\begin{thm}\label{pequals3}
Let $G$ be a finite group of order divisible by $p$ where $p \in \{ 2, 3 \}$. Let $P\in \syl p G$. Then
$$|\irrp p  {B_0(G)}^{\sigma_1}|=p \text{ \ \  if, and only if, \ \ } P \text { is cyclic.}$$
\end{thm}
\begin{proof}
If $P$ is cyclic, then $|\irrp p  {B_0(G)}^{\sigma_1}|=p$ by Lemma \ref{thmbound}.

\smallskip

We assume now that $|\irrp p  {B_0(G)}^{\sigma_1}|=p$ and we work 
to prove that $P$ is cyclic by induction on the order of $G$. 

First, notice that we may assume that $G$ is not simple, by Theorem \ref{Mandi3}(a),
and $\norm G P <G$ by Theorem \ref{thmnormaldefect}.

\smallskip

{\bf Step 1.} \textit{We may assume $\oh{p'} G=1$.}  This follows by Lemma \ref{basicsprincipalblock}(a) and induction.
  
  \smallskip

{\bf Step 2.} \textit{We may assume that $\Oh{p'} G =G$.}
Otherwise, let $M\nor G$ with $|G/M|$ not divisible by $p$ and $G/M>1$ simple. Then $P\sbs M$ and by the Frattini argument $M\norm G P=G$. Hence $M\cent G P\nor G$ and therefore $G=\cent G P M$ or $\cent G P \sbs M$. Suppose $G=M\cent G P$, then restriction defines a bijection ${\rm Irr}_{p'}(B_0(G))^{\sigma_1}\to{\rm Irr}_{p'}(B_0(M))^{\sigma_1}$ by Theorem \ref{isomblocks}. In this case we are done by induction.
Therefore we may assume that $\cent G P \sbs M$. We claim that $B_0(G)$ is the only block of $G$ covering $B_0(M)$. Indeed, let $B$ be a block of $G$ covering $B_0(M)$. By \cite[Theorem 9.26]{Nav98}, we have  that $P$ is a defect group of $B$. By \cite[Lemma 9.20]{Nav98}, $B$ is regular with respect to $M$ and hence by \cite[Theorem 9.19]{Nav98}, $B_0(M)^G=B$. By Brauer's third main theorem we have that $B_0(M)^G=B_0(G)$ and hence $B=B_0(G)$ and the claim is proven. 
In particular, ${\rm Irr}(G/M)\sbs{\rm Irr}_{p'}(B_0(G))^{\sigma_1}$ as every character in $\irr{G/M}$ has $p'$-degree and is $\sigma_1$-invariant (for $G/M$ is a $p'$-group). 
By hypothesis $|\irr{G/M}|\leq p$. As $G/M$ is a nontrivial $p'$-group, we immediately get a contradiction if $p=2$. If $p=3$, then $|\irr{G/M}|\leq 3$ forces
$G/M={\sf C}_2$. Write ${\rm Irr}_{p'}(B_0(G))^{\sigma_1}=\{1,\lambda,\theta\}$ with $M\sbs\ker\lambda$, for instance. Let $\tau\in{\rm Irr}_{p'}(B_0(M))^{\sigma_1}$ be nontrivial by Lemma \ref{divisible_by_p}. Let $\chi\in{\rm Irr}(B_0(G))$ be over $\tau$. Since $|G/M|$ is not divisible by $p$, we have that $\chi\in{\rm Irr}_{p'}(B_0(G))^{\sigma_1}$ by \cite[Lemma 5.1]{NT19}. Thus necessarily $\chi=\theta$. Since $\theta_M$ has at most two irreducible constituents, we have that $|{\rm Irr}_{p'}(B_0(M))^{\sigma_1}|=3$ and we are done by induction in this case. 

\smallskip

{\bf Step 3.} \textit{If $1\neq M\nor G$, then every $\chi \in \irrp p{B_0(G)}^{\sigma_1}$ satisfies $M\sbs \ker \chi$ and $PM/M>1$ is cyclic.}
By Step 2, we have that $p$ divides $|G/M|$ and hence $|\irrp p {B_0(G/M)}^{\sigma_1}|=p$
follows from Lemma \ref{divisible_by_p}. The claim of the step now follows from Lemma \ref{basicsprincipalblock}(a) and by induction.

\smallskip

{\bf Step 4.} \textit{If $1\neq M\nor G$ and $\gamma \in {\rm Irr}_{p'}(B_0(MP))^{\sigma_1}$,
 then there is some $\chi \in {\rm Irr}_{p'}(B_0(G))^{\sigma_1}$ lying over $\gamma$.}
Write $H=MP\cent GP$, so that $MP\nor H$. By Theorem \ref{isomblocks}, restriction defines a bijection ${\rm Irr}_{p'}(B_0(H))^{\sigma_1} \rightarrow {\rm Irr}_{p'}(B_0(MP))^{\sigma_1}$, and hence some $\theta \in {\rm Irr}_{p'}(B_0(H))^{\sigma_1}$ extends $\gamma$. By Lemma \ref{inducing}, the claim of the step follows. 

\smallskip

{\bf Step 5.}  \textit{Let $N$ be a minimal normal subgroup of $G$. We may assume $PN<G$. }
Suppose the contrary.  By Step 1 and the fact that $\norm G P <G$ (so $G$ is not a $p$-group), we have that $N$ is the direct product of $t$ copies of a nonabelian simple group $S$ of order divisible by $p$ (which are transitively permuted by $G$). By Theorem \ref{Mandi3} there exist $1_S\neq\phi\in{\rm Irr}_{p'}(B_0(S))^{\sigma_1}$ $X$-invariant for some $X\in{\rm Syl}_p({\rm Aut}(S))$. By Lemma \ref{invariance_direct_products}, there is some $1_N\neq \theta \in \irrp p{B_0(N)}^{\sigma_1}$ that extends to a $\sigma_1$-invariant character $\chi\in{\rm Irr}(G)$. Since $B_0(G)$ is the only block covering $B_0(N)$ by \cite[Corollary 9.6]{Nav98}, we have that $\chi\in{\rm Irr}_{p'}(B_0(G))^{\sigma_1}$ contradicting Step 3.

\smallskip

 {\bf Final Step.} Since $NP<G$ by Step 5, if $|{\rm Irr}_{p'}(B_0(NP))^{\sigma_1}|=p$, then we are done by induction. Hence we may assume that $|{\rm Irr}_{p'}(B_0(NP))^{\sigma_1}|>p$ by Lemma \ref{divisible_by_p}. By Step 3, we have that $PN/N$ is cyclic, and hence $|{\rm Irr}_{p'}(B_0(PN/N))^{\sigma_1}|=p$. Therefore there exists some $\theta\in{\rm Irr}_{p'}(B_0(NP))^{\sigma_1}$ such that $N\nsubseteq {\rm ker}(\theta)$ (here we are using that $NP/N$ has just one $p$-block). By Step 4, some $\chi\in{\rm Irr}_{p'}(B_0(G))^{\sigma_1}$ lies over $\theta$. In particular $N\nsubseteq \ker \chi$, a contradiction with Step 3.
\end{proof}

\section{Simple groups}\label{simplegroups}

In this Section we prove Theorem \ref{simplegroupsp2}, which will complete the proof of Theorem \ref{thmA}.
\input{rsv_simplesinputArxiv2.tex}

\end{document}

%% file: rsv_simplesinputArxiv2.tex
\subsection{Some Generalities on Groups of Lie Type}\label{generalities}

Since the groups of Lie type play a large role in what follows, we begin by recalling some essentials about their blocks and characters.

Let $q$ be a power of a prime. When $G=\textbf{G}^F$ is the group of fixed points of a connected reductive algebraic group $\textbf{G}$ defined over $\overline{\mathbb{F}}_q$ under a Steinberg map $F$, the set of irreducible characters $\irr G $ can be written as a disjoint union $\bigsqcup \mathcal{E}(G, s)$ of so-called {rational Lusztig series} corresponding to $G^\ast$-conjugacy classes of semisimple elements $s\in G^\ast$ (i.e. elements of order relatively prime to $q$). Here $G^\ast=(\textbf{G}^\ast)^{F^\ast}$, where $(\textbf{G}^\ast, F^\ast)$ is dual to $(\textbf{G}, F)$.     

With this notation, we record the following lemma, proved in \cite[Lemma 3.4]{ST18a}, which describes the action of $\mathcal{H}$ on the set of rational Lusztig series and will be useful throughout this section. 

\begin{lem}\label{SFTlem3.4}
Let $p$ be a prime and let $s\in G^\ast$ be a semisimple element. Let $f$ and $b$ be integers and let $\sigma\in\mathcal{H}$ be such that $\sigma(\xi)=\xi^{p ^f}$  for all $p'$-roots of unity $\xi$ and $\sigma(\zeta)=\zeta^{b}$ for all $p$-power roots of unity $\zeta$.  If $\chi\in\mathcal{E}(G, s)$, then $\chi^\sigma\in\mathcal{E}(G, s_{p'}^{p^f}s_p^{b})$.
\end{lem}

The characters in the series $\mathcal{E}(G, 1)$ are called \emph{unipotent} characters, and there is a bijection $\mathcal{E}(G,s)\rightarrow\mathcal{E}(\cent{G^\ast}{s}, 1)$.  Hence, characters of $\irr G$ may be indexed by pairs $(s, \psi)$, where $s\in G^\ast$ is a semisimple element, up to $G^\ast$-conjugacy, and $\psi\in\irr {\cent{G^\ast}{s}}$ is a unipotent character of ${\cent{G^\ast}{s}}$.  We remark that ${\cent{\bG^\ast}{s}}$ may fail to be connected, in which case unipotent characters of ${\cent{G^\ast}{s}}$ are taken to be those lying over a unipotent character of $({\cent{\bG^\ast}{s}}^\circ)^F$.  In particular, we will denote by $\chi_s$ the the character indexed by $(s, 1_{\cent{G^\ast}{s}})$, which are \emph{semisimple}, and they have degree $|G^\ast:\cent{G^\ast}{s}|_{q'}$.  

Using \cite[Theorem 9.12]{CE04}, it follows that when $p\nmid q$, the set $\mathcal{E}_p(G,1):=\bigcup \mathcal{E}(G, s)$, where $s$ ranges over elements of $p$-power order in $G^\ast$, is a union of $p$-blocks (first shown in \cite{brouemichel}) and that each such block intersects $\mathcal{E}(G,1)$ nontrivially.  Such blocks are called  \emph{unipotent blocks}.

\subsubsection{A General Set-up}\label{sec:setup} We will often be interested in the following situation:

Let $S$ be a simple group such that there exist $\bG$ a simple, simply connected algebraic group over $\overline{\mathbb{F}}_q$ and $F$ a Steinberg morphism satisfying $S=G/\zent{G}$ with $G:=\bG^F$ perfect.  
Let $(\bG^\ast, F^\ast)$ be dual to $(\bG, F)$.

If $\zent{G}$ is trivial, we define $\wt{\bG}:=\bG$.  Otherwise, we further let $\iota\colon\bG\hookrightarrow\wt{\bG}$ be a regular embedding as in \cite[15.1]{CE04} and let $\iota^\ast\colon\wt{\bG}^\ast\rightarrow\bG^\ast$ be the corresponding surjection of dual groups. Write $\wt{G}:=\wt{\bG}^F$, $G^\ast:=(\bG^\ast)^{F^\ast}$, and $\wt{G}^\ast:=(\wt{\bG}^\ast)^{F^\ast}$.  We may then find $F$-stable maximally split tori $\textbf{T}$ and $\wt{\textbf{T}}$ for $\bG$ and $\wt{\bG}$, respectively, such that $\textbf{T}\leq \wt{\textbf{T}}$.  Write $T:=\textbf{T}^F$ and $\wt{T}:=\wt{\textbf{T}}^F$. Then $\zent{\wt{\bG}}$ is connected, $G\nor \wt{G}$, and $\zent{\wt{G}}\cap G=\zent{G}$. We will write $\wt{S}:=\wt{G}/\zent{\wt{G}}$, and note that $\aut S$ is generated by $\wt{S}$ and the graph-field automorphisms.  
Further, the (linear) characters of $\wt{G}/G$ are in bijection with elements of $\zent{\wt{G}^\ast}$, and we have $\wt{\chi}_s\otimes \wh{z}=\wt{\chi}_{sz}$, where $z\in \zent{\wt{G}^\ast}$ corresponds to $\wh{z}\in\irr{\wt{G}/G}$ and for semisimple $s\in\wt{G}^\ast$, $\wt{\chi}_s$ denotes the semisimple character of $\wt{G}$ corresponding to $s$. (See  \cite[13.30]{dignemichel}.)  It will also be useful in what follows to note that if $s\in[\wt{G}^\ast, \wt{G}^\ast]$ is semisimple, then the semisimple character of $\wt{G}$ corresponding to $s$ is trivial on $\zent{\wt{G}}$  by \cite[Lemma 4.4]{navarrotiep13}.

When $q$ is a power of $p$, we note that $\irrp p{B_0(S)}=\irrp p{S}$, which can be seen using \cite[6.14, 6.15, and 6.18]{CE04} and the facts that $p\nmid |\zent{G}|$ and $S$ is a group with a strongly split BN pair as in \cite[2.20]{CE04}.  

In the case of types $A_{n-1}$ and $\tw{2}A_{n-1}$, we have $S$ is $PSL_n^\epsilon(q)$ with $\epsilon\in\{\pm1\}$; $G=SL_n^\epsilon(q)$; $\wt{G}=GL_n^\epsilon(q)$; and $\wt{S}=PGL_n^\epsilon(q)$. Here $\epsilon=1$ means $S=PSL_n(q)$, $\epsilon=-1$ means $S=PSU_n(q)$, and similarly for $G$ and $\wt{G}$.  We use similar notation for other twisted types.  For example, $E^\epsilon_6(q)$ will denote $E_6(q)$ for $\epsilon=+$ and $\tw{2}E_6(q)$ for $\epsilon=-$.


%

\subsection{The Case $p=2$}\label{simples2}

Here we prove Theorem \ref{simplegroupsp2} in the case $p=2$.
%
%
The following, found in 
\cite[Lemma 3.1]{NST18}, will be useful in what follows.

\begin{lem}[Navarro-Sambale-Tiep]\label{nstlem}
Let $G$ be a finite group.
 If $\chi\in\irrp 2{G}$ is real-valued, then $\chi$ belongs to $B_0(G)$.
\end{lem}

 In particular, note that an odd character degree of $G$ with multiplicity one must necessarily come from a character fixed by all automorphisms and $\mathcal{G}$, which is therefore an $X$-invariant member of $\irrp 2 {B_0(G)}^{\sigma_1}$.


\begin{lem}\label{sporadic}
Let $S$ be a simple sporadic group, alternating group $\mathfrak{A}_n$ with $n\geq 5$, or one of the simple groups $PSL_2(4), PSL_3(2), PSL_3(4),$ $PSU_4(2)$, $PSU_4(3)$, $PSL^\epsilon_6(2)$, $\tw{2}B_2(8)$, $B_3(2)$, $B_3(3)$, $D_4(2)$, $F_4(2)$, $\tw{2}F_4(2)'$, $E_6(2)$, $\tw{2}E_6(2)$, $G_2(2)'$, $G_2(3)$, or $G_2(4)$. Then Theorem \ref{simplegroups2} holds for $S$ and the prime $p=2$.
\end{lem}

\begin{proof}
For $n\geq 7$, the automorphism group of $\mathfrak{A}_n$ is the symmetric group $\mathfrak{S}_n$.  Recall that every irreducible character of $\mathfrak{S}_n$ is rational-valued and that an odd-degree character of $\mathfrak{S}_n$ must restrict irreducibly to $\mathfrak{A}_n$ since it has index $2$.  In this case, if $n=2^{n_1}+\cdots+2^{n_t}$ with $n_1<n_2<...<n_t$ is the $2$-adic decomposition of $n$, then \cite[Corollary 1.3]{macdonald71} yields that there are $2^{n_1+\cdots+n_t}\geq 8$ odd-degree characters of $\mathfrak{S}_n$, whose restrictions therefore yield at least $3$ nontrivial members of $\irrp 2{B_0(\mathfrak{A}_n)}^{\sigma_1}$ invariant under $\aut {\mathfrak{A}_n}$. For the remaining groups listed, the statement can be seen using \cite{GAP} and the GAP Character Table Library.  In fact, we see that for the sporadic groups other than the Janko groups, there exist at least two nontrivial odd character degrees with multiplicity 1.
\end{proof}


\begin{pro}\label{liecross}
Let $S$ be a simple group of Lie type defined over $\mathbb{F}_q$ with $q$ a power of an odd prime $\ell$. Then Theorem \ref{simplegroups2} holds for $S$ and the prime $p=2$.
\end{pro}
\begin{proof}
We may assume that $S$ is not isomorphic to any of the groups in Lemma \ref{sporadic}, so is as in Section \ref{sec:setup}.  In this case, the Steinberg character is rational-valued and $\aut{S}$-invariant, and therefore it suffices to show that there is another nontrivial member of $\irrp 2{B_0(S)}^{\sigma_1}$.  Further, we note that if $S$ is not a Suzuki or Ree group, then unipotent characters of odd degree are rational-valued (see, e.g. \cite[Lemma 4.4]{SF19}).  Hence in these cases, applying Lemma \ref{nstlem}, it suffices to find another nontrivial unipotent character of odd degree, when possible.  By observing the explicit list of unipotent character degrees in \cite[Section 13.9]{carter}, we see that there is at least one other nontrivial odd-degree unipotent character for the exceptional groups $G_2(q), \tw{3}D_4(q), F_4(q), E_6^\epsilon(q), E_7(q),$ and $E_8(q)$.  For $\tw{2}G_2(q)$, we see from the generic character table in \cite{chevie} that there is another odd degree with multiplicity one. 

For the classical groups $A_{n-1}(q), \tw{2}A_{n-1}(q), B_n(q), C_n(q), D_n(q),$ or $\tw{2}D_n(q)$, we know by \cite[Proposition 7.4]{MS16} that all unipotent characters of $G$ with odd degree lie in the principal series, and hence are in bijection with the odd-degree irreducible characters of the Weyl group $W$ of $G$.   In these cases, $W$ contains a quotient isomorphic to $\mathfrak{S}_n$, which has at least $4$ odd-degree characters for $n\geq 4$, again using \cite[Corollary 1.3]{macdonald71}. We also see, for example using the GAP, that there are at least 4 odd-degree characters of $W$ in the case of $B_2$, $B_3$, and $C_3$.  Using the well-known character table for $PSL_2(q)$, we see that all four odd-degree characters are fixed by $\sigma_1$.  Further, in this case, $\irrp 2{S}=\irrp 2 {B_0(S)}$.  We see from part (iii) of the proof of \cite[Theorem 3.3]{NST18} that if $S=PSL^\epsilon_3(q)$, then the Weil character $\zeta_{3,q}^{(q-\epsilon)/2}$ is a member of $\irrp 2{B_0(S)}$ and is real-valued.
\end{proof}

\begin{pro}\label{liedef}
Let $S$ be a simple group of Lie type defined in characteristic $2.$ Then Theorem \ref{simplegroups2} holds for $S$ for the prime $p=2$.
\end{pro}

\begin{proof}
Again, we may assume that $S$ is not as in Lemma \ref{sporadic}.  In particular, we may keep the notation as in Section \ref{sec:setup} 
and we have $\irrp 2{B_0(S)}=\irrp 2{S}$.   If $S$ is $\tw{2}B_2(q)$ or $\tw{2}F_4(q)$, then the generic character tables available in CHEVIE yield the result, since $|\out{S}|$ is odd and there are at least two distinct degrees of nontrivial odd-degree characters whose values are fixed by $\sigma_1$.

  Otherwise, we may take the Steinberg endomorphism on $\bG$ to be $F=F_q\circ \tau$, where $F_q$ is the standard Frobenius induced by the map $x\mapsto x^q$ and $\tau$ is some graph automorphism.  Write $\overline{q}:=q^{|\tau|}=2^{2^bm}$ with $m$ odd and let $X\leq \aut{S}$ such that $X/S\in\mathrm{Syl}_2(\out{S})$.  
  

Since $q$ is a power of $2$, we have $\zent{G}=1$ and $\widetilde{G}=G$ unless $S$ is one of $PSL_n^\epsilon(q)$ or $E_6^\epsilon(q)$.  In the latter cases, $G=[\widetilde{G}, \widetilde{G}]$.  
In any case, since $\widetilde{G}/G$ has odd order, we may view $X/S$ as generated by $F_2^m$ and graph automorphisms.

Now, if $m>1$, then the proof of \cite[Lemma 6.4]{ST18a} (and taking into account the remark after \cite[Proposition 6.5]{ST18b}) yields a member of $\irrp 2{S}$ invariant under $X$ which is the restriction to $G$ of a semisimple character of $\widetilde{G}$ trivial on $\zent{G}$.  Since semisimple elements have odd order and $\sigma_1$ fixes odd roots of unity, Lemma \ref{SFTlem3.4} shows that this character is also fixed by $\sigma_1$. If $m=1$, we may similarly obtain an $X$-invariant member of $\irrp 2{S}$ fixed by $\sigma_1$ by arguing as in \cite[Lemma 6.4]{ST18a} and the remark after \cite[Proposition 6.5]{ST18b} but using an element of $\mathbb{F}_{4}^\times$ of order $3$ rather than an element of $\overline{\mathbb{F}}_q^\times$of order $5$.  

  For $S=G_2(q), F_4(q), \tw{3}D_4(q)$, $E_7(q)$, or $E_8(q)$, the list of character degrees at \cite{luebeckwebsite} yields at least one more distinct nontrivial odd character degree, completing the proof in these cases, since by \cite[Theorem 6.8]{malle07}, odd-degree characters are semisimple (recall that we may assume $q\neq 2$ when $S=G_2(q)$ or $F_4(q)$), and hence fixed by $\sigma_1$ using Lemma \ref{SFTlem3.4}. 

Now, in the remaining cases, $S$ is a classical group or $E_6^\epsilon(q)$.  Here $\widetilde{G}^\ast\cong \widetilde{G}$.   In the case $S=PSL_2(q)$ or $PSL_3^\epsilon(q)$, we see that there is at least one more odd-degree character with a different degree that is fixed by $\sigma_1$, using the generic character tables available in \cite{chevie}.  If $\widetilde{G}=GL_n^\epsilon(q), Sp_{n}(q),$ or $\Omega_{n}^\epsilon(q)$ with $n\geq 4$ and $n$ even in the latter two cases, let $s_1$ and $s_2$ be elements of $\widetilde{G}$ with eigenvalues $\{\delta, \delta^{-1}, 1, \ldots, 1\}$ and $\{\delta, \delta, \delta^{-1}, \delta^{-1}, 1,\ldots, 1\}$, respectively, where $1\neq\delta\in\mathbb{F}_{\overline{q}}^\times$.

Then $s_1$ and $s_2$ are not $\aut{S}$-conjugate, and hence the corresponding semisimple characters of $\widetilde{G}$ have odd degree, are not $\aut{S}$-conjugate, and are fixed by $\sigma_1$ by Lemma \ref{SFTlem3.4}.  Further, if $\widetilde{G}=GL_n^\epsilon(q)$, semisimple classes are determined by the eigenvalues, and $\zent{\widetilde{G}}$ is comprised of scalar matrices, so we see for $i=1,2$, $s_i$ is not conjugate to $s_iz$ for any $z\in \zent{\widetilde{G}}$ unless possibly if $n=6$. In this case, we may assume $q\neq 2$ using Lemma \ref{sporadic} and instead take $\delta\in\mathbb{F}_{q^2}^\times$ to have order at least $5$, again yielding $s_i$ is not conjugate to $s_iz$ for any $z\in \zent{\widetilde{G}}$.  In any case, the corresponding semisimple characters therefore restrict irreducibly to $G$ and are trivial on $\zent{\widetilde{G}}$ since $s_i\in[\widetilde{G}, \widetilde{G}]=G$.  Finally, let $S$ be $E_6^\epsilon(q)$ with $q>2$.   Then we may argue analogously to \cite[Proposition 4.3]{GRS} to find elements $s_1$ and $s_2$ in $\widetilde{G}$ with $|\cent{\widetilde{G}}{s_1}|_2\neq |\cent{\widetilde{G}}{s_2}|_2$ such that the corresponding semisimple characters (which again must be fixed by $\sigma_1$) are irreducible on $G$ and trivial on $\zent{\widetilde{G}}$.  (Indeed, we may replace the $\delta$ used there with a $\delta\in\mathbb{F}_{q^2}^\times$ such that $3\nmid |\delta|$).  In all cases, this yields at least one more nontrivial member of $\irrp 2{B_0(S)}^{\sigma_1}$ that is not $\aut{S}$-conjugate to the $X$-invariant one from above.
\end{proof}

Theorem \ref{simplegroups2} for $p=2$ now follows by combining Propositions \ref{liecross} and \ref{liedef} with Lemma \ref{sporadic}.
\subsection{The Case $p=3$}


Here we prove Theorem \ref{Mandi3} in the case $p=3$.   We begin by stating the following classification of simple groups with cyclic Sylow $3$-subgroups.
 
\begin{pro}\label{cyclicsylowclass}
Let $S$ be a finite nonabelian simple group with order divisible by $3$.  Then $S$ has a cyclic Sylow $3$-subgroup if and only if $S$ is one of: the alternating group $\mathfrak{A}_5$; the sporadic simple group $J_1$; $PSL_2(q)$ for $3\nmid q$; or $PSL^\epsilon_3(q)$ for $3\mid (q+\epsilon)$.

\end{pro}

\begin{proof}
The main result of \cite{shenzhau} yields a classification of simple groups $S$ and primes $p$ such that $S$ has an abelian Sylow $p$-subgroup.  In particular, if $p=3$, then such a simple group must be of the form $\mathfrak{A}_n$ with $n<9$, one of a short list of sporadic simple groups, $PSL_2(q)$, $PSL^\epsilon_n(q)$ for $3\mid(q+\epsilon)$ and $n=3,4,5$, or $PSp_4(q)$ with $3\nmid q$.  

Using the Atlas \cite{atlas} and since $\mathfrak{A}_6$ has a noncyclic Sylow $3$-subgroup and can be viewed as a subgroup of $\mathfrak{A}_n$ for $n\geq 7$, we see that the only simple alternating or sporadic groups with cyclic Sylow $3$-subgroups are $\mathfrak{A}_5$ and the Janko group $J_1$.  The remaining possibilities are of the form $G/\zent{G}$ for $G$ a classical group $SL^\epsilon_n(q)$ with $n<6$, or $Sp_4(q)$.  Further, except in the cases of $PSL^\epsilon_3(q)$ listed in the statement, $|\zent{G}|$ is relatively prime to $3$, and hence $S$ has a cyclic Sylow $3$-subgroup if and only if $G$ does.  Further, for the cases $G=SL^\epsilon_n(q)$ with $n=3,4,5$ under consideration, we may view the Sylow subgroup as a Sylow subgroup of $\wt{G}=GL^\epsilon_n(q)$, since $[\wt{G}:G]$ is not divisible by $3$.

Now, using the description of the Sylow subgroups of classical groups in \cite{carterfong, weir}, we see that the Sylow subgroups of $GL^\epsilon_4(q), GL^\epsilon_5(q)$, and $Sp_{4}(q)$ are direct products of Sylow subgroups of at least two lower-rank groups, and hence the Sylow $3$-subgroup of $G$ is not cyclic.  In the case $PSL^\epsilon_3(q)$ with $3\mid (q+\epsilon)$ or $PSL_2(q)$ with $3\nmid q$, we may explicitly construct a cyclic Sylow $3$-subgroup.  Finally, if $S=PSL_2(q)$ with $3\mid q$, the Sylow $3$-subgroup can be identified with the unipotent radical of $SL_2(q)$, which is not cyclic unless $q=3$, contradicting that $S$ is simple. 
\end{proof}

Our goal in the remainder of this section is to prove the following, from which we obtain Theorem \ref{Mandi3} for $p=3$ as a corollary.

\begin{thm}\label{simples3}
Let $S$ be a nonabelian simple group with order divisible by $3$. 
\begin{enumerate}
\item If $S$ has a cyclic Sylow $3$-subgroup, then there exist $1_S\neq \chi_1, \chi_2 \in\irrp 3{B_0(S)}^{\sigma_1}$ such that $\chi_1$ extends to $\aut{S}$. 
\item If $S$ does not have a cyclic Sylow $3$-subgroup and is not a group of Lie type defined in characteristic $3$, then there exist nontrivial $\chi_1, \chi_2, \chi_3\in\irrp 3{B_0(S)}^{\sigma_1}$ such that $\chi_1$ and $\chi_2$ extend to $\aut S$. In this case, if $S$ is further not one of $\mathfrak{A}_6, \mathfrak{A}_7$, $\tw{2}F_4(2)'$, $PSL_n(q)$ with $n\leq 4$, or $PSp_4(2^{2m+1})$, then there exist nontrivial $\chi_1, \chi_2, \chi_3\in\irrp 3{B_0(S)}^{\sigma_1}$ such that $\chi_i$ each extend to $\aut{S}$.  
\item If $S$ is a group of Lie type in characteristic $3$, then there exist nontrivial $\chi_1, \chi_2, \chi_3\in\irrp 3{B_0(S)}^{\sigma_1}$ that are pairwise not $\aut{S}$-conjugate and such that $\chi_1$ is invariant under $X$, where $X/S\in\mathrm{Syl}_3(\aut S/S)$.  
\end{enumerate}
\end{thm}


We first consider Theorem \ref{simples3} for sporadic and alternating groups, as well as some ``small" groups of Lie type. For two positive integers $n$ and $m$, we will use $n\mid\mid m$ to mean that $n\mid m$ and $\gcd(n, \frac{m}{n})=1$.

\begin{pro}\label{sporadicalt3}
Theorem \ref{simples3} holds for the sporadic simple groups, $G_2(3)$, $\tw{2}F_4(2)'$, $B_3(3)$, $G_2(2)'=PSU_3(3)$, $PSU_4(3)$, and the alternating groups $\mathfrak{A}_n$ with $n\geq 5$. 
\end{pro}

\begin{proof}
Since the result can be seen directly using GAP for the other cases, we may assume $S=\mathfrak{A}_n$ with $n>10$.  In this case, $S$ does not have a cyclic Sylow $3$-subgroup and satisfies $\aut{S}=\mathfrak{S}_n$, where $\mathfrak{S}_n$ denotes the corresponding symmetric group.

The characters of $\mathfrak{S}_n$ are rational-valued and parametrized by partitions of $n$, with their degrees given by the hook formula.  Further, two characters lie in the same $3$-block if and only if they have the same $3$-core.  We also know that $\chi\in\irr{\mathfrak{S}_n}$ corresponding to the partition $\lambda$ restricts irreducibly to $\mathfrak{A}_n$ if and only if the partition is not self-conjugate.  Table \ref{tab:An} lists the partitions and character degrees for three characters in $\irrp 3{\mathfrak{S}_n}$ that restrict irreducibly to $\mathfrak{A}_n$, completing the proof.  \end{proof}

\begin{table}
\caption{Some Members of $\irrp 3 {B_0(\mathfrak{S}_n)}$ Irreducible on $\mathfrak{A}_{n}$, $n> 10$}
\begin{tabular}{|c|c|c|}
\hline
 Condition on $n$ & Partition & $\chi(1)$\\
\hline
$3\mid n$ & $(1, n-1)$ & $n-1$\\
\hline
$3\mid n$ & $(1, 1, n-2)$ & $(n-1)(n-2)/2$\\
\hline
$3\mid\mid n, 3^2\mid (n-2),$ or $3\mid\mid (n-1)$ & $(3,n-3)$ & $n(n-1)(n-5)/6$\\
\hline
$3^2|n, 3\mid\mid(n-2)$, or $3\mid\mid (n-1)$ & $(1^3, n-3)$ & $(n-1)(n-2)(n-3)/6$\\
\hline
$3\mid (n-1)$ & $(2, n-2)$ & $n(n-3)/2$\\
\hline
$3^2\mid (n-1)$ & $(1, 2, n-3)$ & $n(n-2)(n-4)/3$\\
\hline
$3^2\mid (n-1)$ & $(1^3, 2, n-5)$ & $n(n-2)(n-3)(n-4)(n-6)/30$\\
\hline
$3\mid (n-2)$ & $(1^{n-4}, 2, 2)$ & $n(n-3)/2$\\
\hline
$3\mid(n-2)$ & $(1^{n-2}, 2)$ & $n-1$\\
\hline 
\end{tabular}\label{tab:An}
\end{table}

\subsubsection{Lie Type in Cross-characteristic for $p=3$}

In this section, we prove Theorem \ref{simples3} for groups of Lie type in non-defining characteristic.  That is, we deal with the case $S$ is of the form $G/\zent{G}$ for $G$ a finite group of Lie type  of simply connected type defined over a field $\mathbb{F}_q$ with $3\nmid q$. (Given Proposition \ref{sporadicalt3}, this will complete the proof of parts (i) and (ii) of Theorem \ref{simples3}.)

  We will use $\Phi_m$ to denote the $m$th cyclotomic polynomial in the variable $q$.   Note that using e.g. \cite[Lemma 5.2]{malle07}, $3$ divides $\Phi_m$ if and only if $m=3^id$ for some $i\geq 0$, where $d$ is the order of $q$ modulo $3$, and in this case $3\mid\mid\Phi_m$ unless $m=d$.

\begin{pro}\label{lietypecross}
 Let $S$ be a simple group of Lie type defined over $\mathbb{F}_q$ with $3\nmid q$ and assume $S$ is not one of the groups $PSL^\epsilon_n(q)$ with $n\leq 3$. Then there exist three nontrivial characters $\chi_1, \chi_2, \chi_3\in\irrp 3{B_0(S)}^{\sigma_1}$ such that $\chi_1$ and $\chi_2$ extend to $\aut{S}$.
 
 Further, if $S$ is not $PSL^\epsilon_4(q)$ nor $PSp_{4}(2^a)$ with $a$ odd, then $\chi_1, \chi_2,$ and $\chi_3$ may be chosen to extend to $\aut{S}$.
\end{pro}
\begin{proof}

We may assume that $S$ is not isomorphic to one of the groups considered in Proposition \ref{sporadicalt3}. Keep the notation and considerations for $G$, $\wt{G}$, $\wt{T}$, and $\wt{S}$ from Section \ref{generalities}.  By the work of Lusztig \cite{lusztig}, the unipotent characters of $\wt{G}$ are trivial on $\zent{\wt{G}}$ and restrict irreducibly to $G$.  Further, when viewed as characters of $\wt{S}$, they are extendible to $\aut S$, by \cite[Theorems 2.4 and 2.5]{malle08}, aside from some specific exceptions. The only unipotent characters which take irrational values occur for exceptional groups and have values in $\mathbb{Q}(\sqrt{-1}, \zeta_3, \zeta_5, \sqrt{q})$, where $\zeta_3$ and $\zeta_5$ are $3$rd and $5$th roots of unity, respectively, by \cite[Prop 5.6 and Table 1]{geck03}. In any case, the unipotent characters are $\sigma_1$-invariant, since $\sqrt{q}$ is a sum of roots of unity of order relatively prime to $3$.

Let $d$ be the order of $q$ modulo $3$. In particular, we have $d=1$ or $2$.  If $d=1$, unipotent characters of degree relatively prime to $3$ are constituents of the Harish-Chandra induced character $R_{\wt{T}}^{\wt{G}}(1)$ using \cite[Corollary 6.6]{malle07}.  Further, by \cite[Theorem A]{enguehard00}, all members of $R_{\wt{T}}^{\wt{G}}(1)$ lie in the same block, namely $B_0(\wt{G})$.  

If $d=2$, 
then the centralizer of a Sylow $d$-torus is a maximal torus, using e.g. \cite[Lemma 3.2]{MS16}.  Unipotent blocks of $\wt{G}$ are parametrized by certain $\wt{G}$-conjugacy classes of pairs $(\wt{L}, \lambda)$ where $\wt{L}$ is a $d$-split Levi subgroup of $\wt{G}$ and $\lambda$ is a $d$-cuspidal unipotent character of $\wt{L}$, by \cite[Theorem A]{enguehard00}.  Further, a unipotent character in the block parametrized by $(\wt{L}, \lambda)$ can have $3'$-degree only when $\wt{L}$ is the centralizer of a Sylow $d$-torus, using \cite[Corollary 6.6]{malle07}.  This yields that again in the case $d=2$, there is a unique block of $\wt{G}$ containing unipotent characters of $3'$-degree. 

Hence when $3\nmid q$, every unipotent character in $\irrp 3{\wt{G}}$ is a member of $\irrp 3{B_0(\wt{G})}^{\sigma_1}$, and restricts to a member of $\irrp 3{B_0(G)}^{\sigma_1}$ trivial on the center.  Then this restriction may be viewed as an element of $\irrp 3{B_0(S)}^{\sigma_1}$, using e.g. \cite[Lemma 17.2]{CE04}.

In particular, since the Steinberg character has degree a power of $q$, it suffices to find two more unipotent characters of $3'$-degree that are not one of the exceptional cases in \cite[Theorem 2.5]{malle08}.  In what follows, we will use the notation and degrees for unipotent characters as in \cite[Sections 13.8 and 13.9]{carter}.  


\vspace{.3cm}

\paragraph{\textbf{Exceptional Types}}\label{crossexcept}
In the case that $S$ is an exceptional group of Lie type defined over $\mathbb{F}_q$ with $3\nmid q$, we list in Table \ref{tab:unipexcep} two unipotent characters invariant under $\aut{S}$ that have degree relatively prime to $3$, completing the proof in this case.   

\begin{table}\small
\begin{center} 
\caption{Some Unipotent Characters in $\irrp 3 {B_0(S)}^{\sigma_1}$ for Exceptional types with $3\nmid q$ }\label{tab:unipexcep}
\begin{tabular}{|c|c|c|c|}
\hline
\multirow{2}{*}{Type} & \multirow{2}{*}{Condition on $d$} & Character  & \multirow{2}{*}{$\chi(1)$}\\
& & (Notation from \cite[13.9]{carter}) & \\
\hline
\multirow{4}{*}{$G_2(q)$} & \multirow{2}{*}{$d=1$} & $\phi_{2,2}$ & $\frac{1}{2}q\Phi_2^2\Phi_6$\\
 & & $\phi_{1,3'}$ & $\frac{1}{3}q\Phi_3\Phi_6$\\
 \cline{2-4}

 & \multirow{2}{*}{$d=2$} & $G_2[1]$ & $\frac{1}{6}q\Phi_1^2\Phi_6$\\
 & & $\phi_{1,3'}$ & $\frac{1}{3}q\Phi_3\Phi_6$\\
\hline
\multirow{2}{*}{$\tw{3}{D_4}(q)$} & \multirow{2}{*}{$d=1,2$} & $\phi_{1,3'}$ & $q\Phi_{12}$ \\
 & & $\phi_{1,3''}$ & $q^7\Phi_{12}$ \\
 \hline
 \multirow{4}{*}{${F_4}(q)$} & \multirow{2}{*}{$d=1$} & $\phi_{4,1}$ & $\frac{1}{2}q\Phi_2^2\Phi_6^2\Phi_8$ \\
 & & $\phi_{8,3'}$ & $q^3\Phi_4^2\Phi_8\Phi_{12}$ \\
 \cline{2-4}
  & \multirow{2}{*}{$d=2$} & $B_{2, \epsilon}$ & $\frac{1}{2}q^{13}\Phi_1^2\Phi_3^2\Phi_8$ \\
 & & $B_{2,1}$ & $\frac{1}{2}q\Phi_1^2\Phi_3^2\Phi_8$ \\
\hline
\multirow{2}{*}{${E_6}(q)$} & \multirow{2}{*}{$d=1, 2$} & $\phi_{20, 2}$ & $q^2\Phi_4\Phi_5\Phi_8\Phi_{12}$ \\
 & & $\phi_{20, 20}$ & $q^{20}\Phi_4\Phi_5\Phi_8\Phi_{12}$ \\
 \hline
 \multirow{2}{*}{${\tw{2}{E}_6}(q)$} & \multirow{2}{*}{$d=1, 2$} & $\phi_{4, 1}$ & $q^2\Phi_4\Phi_8\Phi_{10}\Phi_{12}$ \\
 & & $\phi_{4, 13}$ & $q^{20}\Phi_4\Phi_8\Phi_{10}\Phi_{12}$ \\
 \hline
  \multirow{2}{*}{${{E}_7}(q)$} & \multirow{2}{*}{$d=1, 2$} & $\phi_{7, 1}$ & $q\Phi_7\Phi_{12}\Phi_{14}$ \\
 & & $\phi_{7, 46}$ & $q^{46}\Phi_7\Phi_{12}\Phi_{14}$ \\
\hline
  \multirow{2}{*}{${{E}_8}(q)$} & \multirow{2}{*}{$d=1, 2$} & $\phi_{8, 1}$ & $q\Phi_4^2\Phi_8\Phi_{12}\Phi_{20}\Phi_{24}$ \\
 & & $\phi_{35, 2}$ & $q^{2}\Phi_5\Phi_7\Phi_{10}\Phi_{14}\Phi_{15}\Phi_{20}\Phi_{30}$ \\
 \hline
 {$\tw{2}{B}_2(q)$,} & \multirow{2}{*}{Note: $3\nmid |S|$} & $\tw{2}B_2[a]$ & $\frac{1}{\sqrt{2}}q(q^2-1)$\\
 $q^2=2^{2m+1}$ & &$\tw{2}B_2[b]$ & $\frac{1}{\sqrt{2}}q(q^2-1)$\\
\hline
 {$\tw{2}{F}_4(q)$,} & \multirow{2}{*}{Note: $3\mid(q^2+1)$} & cusp & $\frac{1}{12}q^4\Phi_1^2\Phi_2^2(\Phi_8')^2\Phi_{12}(\Phi_{24}')^2$\\
 $q^2=2^{2m+1}$ & &cusp & $\frac{1}{12}q^4\Phi_1^2\Phi_2^2(\Phi_8'')^2\Phi_{12}(\Phi_{24}'')^2$\\

\hline \hline
\end{tabular}
\end{center}
\end{table}

\vspace{.3cm}

\paragraph{\textbf{Types $A_{n-1}$ and $\tw{2}{A}_{n-1}$, $n\geq 4$}}

In this case, let $S$ be $PSL_n^\epsilon(q)$ with $n\geq 4$ and $\epsilon\in\{\pm1\}$. 
Write $e\in\{1, 2\}$ for the number such that $q\equiv \epsilon e\pmod 3$. That is, $e$ is the order of $\epsilon q$ modulo $3$. Two unipotent characters are in the same $3$-block of $\wt{G}=GL_n^\epsilon(q)$ if and only if they have the same $e$-core (see \cite{FS82}).  For $n\geq 5$ or for $(n,e)=(4,2)$, the unipotent characters described in Table \ref{tab:unipA} are $\aut S$-invariant members of $\irrp 3 {B_0(S)}^{\sigma_1}$. 

\begin{table}\footnotesize
\caption{Some Unipotent Characters in $\irrp 3 {B_0(S)}^{\sigma_1}$ for type $A^\epsilon_{n-1}(q)$ with $n\geq 4$ and $3\nmid q$}\label{tab:unipA}
\begin{tabular}{|c|c|c|c|}
\hline
$n$ & Additional Condition on $n, e$ & Partition & $\chi(1)_{q'}$\\
\hline \hline
\multirow{2}{*}{$n\geq 6$} & $e=2$ and $n$ even; or &  $\left(1, n-1\right)$  & \multirow{2}{*}{$\frac{q^{n-1}-\epsilon^{n-1}}{q-\epsilon}$} \\
& $e=1$ and $3\nmid (n-1)$ & $\left( 1^{n-2}, 2\right)$  & \\
\hline
\multirow{2}{*}{$n\geq 6$}&\multirow{2}{*}{$e|n$ and $3\nmid n$} & $\left(2, n-2\right)$ & \multirow{2}{*}{$\frac{(q^n-\epsilon^n)(q^{n-3}-\epsilon^{n-3})}{(q-\epsilon)(q^2-1)}$} \\
 & & $\left(1^{n-4}, 2, 2\right)$ & \\
 \hline
\multirow{2}{*}{$n\geq 6$} & $3\mid n$; or  & $\left(1, 1, n-2\right)$ &  \multirow{2}{*}{$\frac{(q^{n-1}-\epsilon^{n-1})(q^{n-2}-\epsilon^{n-2})}{(q-\epsilon)(q^2-1)}$} \\
 & $e=2$ and $n$ odd and $3\mid (n-2)$ & $\left(1^{n-3}, 3\right)$ & \\
 \hline
\multirow{2}{*}{$n\geq 6$} & \multirow{2}{*}{$e=2$ and $n$ odd and $3\mid (n-1)$}  & $\left(1, 1, 2, n-4\right)$ &  $\frac{(q^n-\epsilon^n)(q^{n-2}-\epsilon^{n-2})(q^{n-3}-\epsilon^{n-3})(q^{n-5}-\epsilon^{n-5})}{(q-\epsilon)^2(q^2-1)^2(q^2+1)}$ \\
 &  & $\left(2, n-2\right)$ & $\frac{(q^n-\epsilon^n)(q^{n-3}-\epsilon^{n-3})}{(q-\epsilon)(q^2-1)}$\\
 \hline
 \multirow{2}{*}{$n=5$} & \multirow{2}{*}{$e=1$} & $\left(1, 4\right)$ & \multirow{2}{*}{$(q+\epsilon)(q^2+1)$}\\
 & & $\left(1, 1, 1, 2\right)$ & \\
 \hline
  \multirow{2}{*}{$n=5$} & \multirow{2}{*}{$e=2$} & $\left(2, 3\right)$ & {$\frac{q^5-\epsilon}{q-\epsilon}$}\\
 & & $\left(1, 1, 3\right)$ & $(q^2+1)(q^2+\epsilon q+1)$\\
\hline
\multirow{2}{*}{$n=4$} & \multirow{2}{*}{$e=2$} & $\left(1,3\right)$ & \multirow{2}{*}{$q^2+\epsilon q+1$} \\
& & $\left(1,1,2\right)$ & \\
\hline
$n=4$ & $e=1$ & $\left(2,2\right)$ & $q^2+1$ \\
\hline
\end{tabular}
\end{table}

Now assume $n=4$ and $e=1$.  Then the unipotent character in the last line of Table 3 of the ArXiv version of this paper is an $\aut S$-invariant member of $\irrp 3 {B_0(S)}^{\sigma_1}$.  In this case, $1_S, \mathrm{St}_S$, and the character listed are the only unipotent characters in $\irrp 3{S}$. However, since $e=1$, we see that every unipotent character is a member of the principal block of $\wt{G}$, which means that $\mathcal{E}_3(\wt{G}, 1)$ is comprised of only one block. 
Let $\zeta\in\mathbb{F}_{q^2}^\times$ with order $3$.  Then taking $s$ to be the element $\diag(\zeta, \zeta, \zeta, 1)$ of $\wt{G}^\ast\cong GL_4^\epsilon(q)$, the semisimple character $\chi_s \in \mathcal{E}(\wt{G}, s)$ lies in the principal block of $\wt{G}$ and is trivial on $\zent{\wt{G}}$ since $s\in SL_4^\epsilon(q)\cong [\wt{G}^\ast, \wt{G}^\ast]$.  Further, we see using Lemma \ref{SFTlem3.4} that $\chi_s$ is fixed by $\sigma_1$.  

Since $\cent{\wt{G}^\ast}{s}\cong GL^\epsilon_1(q)\times GL^\epsilon_3(q)$, we see $\chi_s(1)=(q+\epsilon)(q^2+1)$. Further, since the semisimple classes of $\wt{G}$ are determined by their eigenvalues and  $\zent{\wt{G}}$ is comprised of scalar matrices, we see that $s$ is not conjugate to $sz$ for any nontrivial $z\in \zent{\wt{G}}$.  Hence $\chi_s|_G$ is irreducible, by the second-to-last paragraph of Section \ref{generalities}, and is therefore a member of $\irrp 3 {B_0(G)}$, since the principal block of $G$ is the only block covered by the principal block of $\wt{G}$.  But since $\zent{G}\leq \zent{\wt{G}}$ is in the kernel of $\chi_s$, this character is therefore a member of   $\irrp 3 {B_0(G/\zent{G})}^{\sigma_1}=\irrp 3 {B_0(S)}^{\sigma_1}$, again using \cite[Lemma 17.2]{CE04}. Note that this character is not $\aut{S}$-conjugate to $1_S, \mathrm{St}_S$, nor the unipotent character labeled by $(2,2)$, which completes the proof for $S=PSL_n^\epsilon(q)$ with $n\geq 4$.
 
\vspace{.3cm}

\paragraph{\textbf{Types $B_n$ and $C_n$, $n\geq 2$}}\label{crossBC}

When $S$ is type $B_n$ or $C_n$ with $n\geq 2$ defined in characteristic different than $3$, Table \ref{tab:unipBC} exhibits at least two distinct unipotent characters in $\irrp 3{B_0(S)}^{\sigma_1}$ that are $\aut{S}$-invariant, with the exception of the case $S=PSp_4(2^{a})$ with $a$ odd.  In the latter situation, we may instead consider the characters indexed by $ 1 \hbox{ } 2 \choose 0 $ and $ 0 \hbox{ } 1 \hbox{ } 2 \choose  \emptyset $  with degrees $\frac{q}{2}(q^2+1)$ and $\frac{q}{2}(q-1)^2$, respectively.  (Note that we do not consider $PSp_4(2)\cong \mathfrak{S}_6$.)  These characters lie in $\irrp 3{B_0(S)}^{\sigma_1}$ and the latter character extends to $\aut{S}$.  (However, we remark that the first character is not $\aut{S}$-invariant, as in this case it is switched with $ 0 \hbox{ } 1 \choose 2 $ under the action of the graph automorphism, by \cite[Theorem 2.5]{malle08}).

\begin{table}\small
\caption{Some Unipotent Characters in $\irrp 3 {B_0(S)}^{\sigma_1}$ for type $B_{n}(q)$ and 
$C_n(q)$ with $n\geq 2$ and $3\nmid q$}\label{tab:unipBC}
\begin{tabular}{|c|c|c|}
\hline
 Conditions on $q, n$ & Symbol & $\chi(1)_{q'}$ (possibly excluding factors of $1/2$)\\
\hline \hline
{$3\mid(q-1)$ or} & \multirow{3}{*}{ $0 \hbox{ } 1 \choose n$} & \multirow{3}{*}{$\frac{(q^{n-1}+1)(q^n+1)}{q+1}$}\\
$3\mid(q+1)$; $n$ even; $3\nmid(n-1)$ or & & \\
 $3\mid(q+1)$; $n$ odd; $3\nmid n$ & & \\
 \hline
$3\mid(q-1); 3\nmid n$ &$ 0 \hbox{ } 2 \choose n-1 $ & $\frac{(q^{2n}-1)(q^{n-3}+1)(q^{n-1}+1)}{q^4-1}$\\
\hline
$3\mid(q-1); 3\nmid (n-1)$ or & \multirow{2}{*}{$1\hbox{ } n\choose 0$} & \multirow{2}{*}{$\frac{(q^{n-1}-1)(q^n+1)}{q-1}$} \\
{$3\mid(q+1)$; $n$ even } & & \\
\hline
$3\mid(n-1)$  & $1 \hbox{ } n-1 \choose 1$ & $\frac{(q^{2n}-1)(q^{2(n-2)}-1)}{(q^2-1)^2}$\\
\hline
{$3\mid(q+1)$; $n$ odd} or  & \multirow{2}{*}{$0\hbox{ } n\choose 1$} & \multirow{2}{*}{$\frac{(q^{n-1}+1)(q^n-1)}{q-1}$} \\
$3\mid(q-1)$; $3\nmid n$ & & \\
\hline
{$3\mid(q+1)$; $n$ odd; $3\mid n$ } or  & \multirow{3}{*}{$0\hbox{ } 1 \hbox{ } n \choose 1 \hbox{ } 2$} & \multirow{3}{*}{$\frac{(q^{2(n-1)}-1)(q^n-1)(q^{n-2}+1)}{(q^2-1)^2}$}\\
$3\mid(q+1)$; $n$ even; $3\mid(n-2)$ or & &\\
$3\mid(q-1)$; $3\mid(n-2)$ & & \\
\hline
\end{tabular}
\end{table}

\vspace{.3cm}

\paragraph{\textbf{Type $D_n$ and $\tw{2}D_n$, $n\geq 4$}}\label{crossD}
In this case, if $S$ is not $D_4(q)$,  Tables \ref{tab:unipD} and \ref{tab:unip2D} list at least two distinct unipotent characters that are $\aut{S}$-invariant members of $\irrp 3{B_0(S)}^{\sigma_1}$.  If $S$ is $D_4(q)$ and $3\mid(q-1)$, we may instead take the unipotent characters labeled by symbols $1 \hbox{ } 3 \choose 0 \hbox{ } 2$ and $3\choose 1$ with $\chi(1)_{q'}=\frac{1}{2}(q+1)^3(q^3+1)$ and $(q^2+1)^2$, respectively. When $3\mid(q+1)$, we may take the characters index by $3\choose 1$ and $0 \hbox{ } 1\hbox{ } 2\hbox{ } 3\choose \emptyset$, the latter of which satisfies $\chi(1)_{q'}=\frac{1}{2}(q-1)^3(q^3-1)$.
\end{proof}

\begin{table}\small
\caption{Some Unipotent Characters in $\irrp 3 {B_0(S)}^{\sigma_1}$ for type $D_{n}(q)$ with $n\geq 5$ and $3\nmid q$}\label{tab:unipD}
\begin{tabular}{|c|c|c|}
\hline
 Conditions on $q, n$ & Symbol & $\chi(1)_{q'}$ (possibly excluding factors of $1/2$)\\
\hline \hline
{$3\nmid (n-1)$} & $1 \hbox{ } n \choose 0 \hbox{ } 1$ & $\frac{q^{2(n-1)}-1}{q^2-1}$\\
\hline
{$3\mid n$} & $1 \hbox{ } 2 \hbox{ } n \choose 0 \hbox{ } 1 \hbox{ } 2$ & $\frac{(q^{2(n-2)}-1)(q^{2(n-1)}-1)}{(q^2-1)^2(q^2+1)}$\\
\hline
{$3\mid (n-2)$} or & \multirow{2}{*}{$n-2\choose 2$} & \multirow{2}{*}{$\frac{(q^{2(n-1)}-1)(q^n-1)(q^{n-4}+1)}{(q^2-1)^2(q^2+1)}$} \\
{$3\mid(q+1); 3\mid n; n \hbox{ odd}$} & & \\
\hline
{$3\mid(q-1); 3\mid (n-1)$} or & \multirow{2}{*}{$1 \hbox{ } n-1 \choose 0 \hbox{ } 2$} & \multirow{2}{*}{$\frac{(q^n-1)(q^{n-2}-1)(q^{n-1}+1)(q^{n-3}+1)}{(q-1)^2(q^2+1)}$}\\
$3\mid(q+1);$ $n$ odd &&\\
\hline
{$3\mid(q-1); 3\mid (n-2)$} or & \multirow{3}{*}{$0 \hbox{ } n-1 \choose 1 \hbox{ } 2$} & \multirow{3}{*}{$\frac{(q^n-1)(q^{n-2}+1)(q^{n-1}-1)(q^{n-3}+1)}{(q^2-1)^2}$}\\

{$3\mid(q+1); 3\nmid n; n \hbox{ even}$} or & & \\
{$3\mid(q+1); 3\mid n; n \hbox{ odd}$} & & \\
 \hline
{$3\mid(q-1); 3\nmid n$} or & \multirow{4}{*}{$n-1  \choose 1$} & \multirow{4}{*}{$\frac{(q^{n}-1)(q^{n-2}+1)}{q^2-1}$}\\
{$3\mid(q+1); 3\mid (n-1)$} or& & \\
{$3\mid(q+1); 3\mid (n-2); n \hbox{ even}$} or& & \\
{$3\mid(q+1); 3\mid n; n \hbox{ odd}$} & & \\
\hline
\end{tabular}
\end{table}

\begin{table}
\caption{Some Unipotent Characters in $\irrp 3 {B_0(S)}^{\sigma_1}$ for type $\tw{2}D_{n}(q)$ with $n\geq 4$ and $3\nmid q$}\label{tab:unip2D}
\begin{tabular}{|c|c|c|}
\hline
 Conditions on $q, n$ & Symbol & $\chi(1)_{q'}$ (possibly excluding factors of $1/2$)\\
\hline \hline
{$3\nmid (n-1)$} & $0 \hbox{ } 1 \hbox{ } n \choose 1$ & $\frac{q^{2(n-1)}-1}{q^2-1}$\\
\hline
{$3\mid (n-1)$} or & \multirow{4}{*}{$ 1 \hbox{ } n-1 \choose \emptyset$} &\multirow{4}{*}{ $\frac{(q^{n}+1)(q^{n-2}-1)}{q^2-1}$}\\
$3\mid(q-1); 3\mid n$ or & & \\
$3\mid(q+1); 3\mid n$; $n$ even or & & \\
$3\mid(q+1); 3\mid (n-2)$; $n$ odd  & & \\
\hline
{$3\mid(q-1); 3\nmid(n-1)$} or & \multirow{3}{*}{$2\hbox{ } n-2\choose \emptyset$} & \multirow{3}{*}{$\frac{(q^{2(n-1)}-1)(q^n+1)(q^{n-4}-1)}{(q^2-1)^2(q^2+1)}$ }\\
{$3\mid(q+1); 3\mid(n-2)$} or &&\\
{$3\mid(q+1); 3\mid n$; $n$ even} &&\\
\hline
$ 3\mid n$ & $0\hbox{ } 1\hbox{ } 2\hbox{ } n\choose 1\hbox{ } 2$ & $\frac{(q^{2(n-1)}-1)(q^{2(n-2)}-1)}{(q^2-1)^2(q^2+1)}$\\
\hline
$3\mid(q-1); 3\mid(n-1)$ & \multirow{3}{*}{$ 1 \hbox{ } 2\hbox{ } n-1 \choose 0$} & \multirow{3}{*}{$\frac{(q^n+1)(q^{n-1}+1)(q^{n-2}-1)(q^{n-3}-1)}{(q^2-1)^2}$} \\
$3\mid(q+1); 3\nmid n; n\hbox{ odd}$ & & \\
$3\mid(q+1); 3\mid n; n\hbox{ even}$ & & \\
\hline
$3\mid(q-1); 3\mid n$ & & \\
$3\mid(q+1); 3\mid(n-1); n\hbox{ even}$ & $ 0 \hbox{ } 1\hbox{ } n-1 \choose 2$ & $\frac{(q^n+1)(q^{n-1}-1)(q^{n-2}-1)(q^{n-3}+1)}{(q^2-1)^2}$\\
$3\mid(q+1); 3\mid(n-2); n\hbox{ odd}$ & & \\
\hline
\end{tabular}
\end{table}

We next establish Theorem \ref{simples3} for the case that $S$ has cyclic Sylow $3$-subgroups, which we recall from Proposition \ref{cyclicsylowclass} occurs when $S=PSL_2(q)$ for $3\nmid q$ and when $S=PSL_3^\epsilon(q)$ for $3\mid (q+\epsilon)$.  
\begin{pro}\label{typeAcyclic}
Let $S=PSL_2(q)$ with $3\nmid q$ or $PSL^\epsilon_3(q)$ with $3\mid(q+\epsilon)$.  Then there exist nontrivial $\chi_1, \chi_2 \in\irrp 3{B_0(S)}^{\sigma_1}$ such that $\chi_1$ extends to $\aut{S}$. 
\end{pro}
\begin{proof}
First let $S=PSL_2(q)$ with $3\nmid q$.  In this case, every character of $S$ is either $3$-defect zero or has degree prime to $3$.  As before, the Steinberg character is a member of $\irrp 3{B_0(S)}^{\sigma_1}$ and extends to $\aut{S}$.  Further,  
the only two unipotent characters, $1_{\wt{G}}$ and $\mathrm{St}_{\wt{G}}$, both lie in the principal block of $\wt{G}=GL_2(q)$, and hence there is a unique unipotent block of $\wt{G}$.  We may take $\chi_1=\mathrm{St}_{S}$ as before.

Now, let $s\in\wt{G}=GL_2(q)$ have eigenvalues $\zeta, \zeta^{-1}$, where $\zeta\in\mathbb{F}_{q^2}^\times$ has order $3$.   Then the semisimple character $\chi_s \in \mathcal{E}(\wt{G}, s)\subseteq \mathcal{E}_3(\wt{G}, 1)$ lies in the principal block of $\wt{G}$ and is trivial on $\zent{\wt{G}}$ since $s\in SL_2(q)\cong [\wt{G}^\ast, \wt{G}^\ast]$.  Since $sz$ is not conjugate to $s$ for $1\neq z\in \zent{\wt{G}^\ast}$, we also see $\chi_s$ is irreducible on restriction to $G$.  Further, Lemma \ref{SFTlem3.4} yields that $\chi_s$ is fixed by $\sigma_1$.  Then the restriction $(\chi_s)_G$ lies in $B_0(G)$ since the principal block of $\wt{G}$ covers a unique block of $G$. Finally, in this case $\chi_s(1)=q+\eta$, where $\eta\in\{\pm1\}$ is such that $3\mid q-\eta$.  Hence this character may be viewed as a member of $\irrp 3{B_0(S)}^{\sigma_1}$, arguing as before.

Now let $S=PSL_3^\epsilon(q)$ with $3\mid(q+\epsilon)$. Then $S=G=SL_3^\epsilon(q)$ and $\wt{G}^\ast\cong\wt{G}=G\times \zent{\wt{G}}$.  Since the unipotent characters of $G$ are $1_G, \mathrm{St}_G$, and a character of degree $q(q+\epsilon)$, we see that again $B_0(G)$ is the only unipotent block of maximal defect (as the other has defect zero).  Then every character of $\mathcal{E}_3(G, 1)$ with $3'$-degree is a member of $B_0(G)$.  We may again take $\chi_1=\mathrm{St}_G$.  Taking $s\in\wt{G}^\ast$ to have eigenvalues $\{\zeta, \zeta^{-1}, 1\}$, where $\zeta\in\mathbb{F}_{q^2}^\times$ has order $3$, the corresponding character of $G$ has degree $q^3-\epsilon$, and we may again view $(\chi_s)_G$ as a character of $\irrp 3{B_0(S)}^{\sigma_1}$.
\end{proof}

\begin{pro}\label{SL3notcyclic}
Let $S=PSL_3^\epsilon(q)$ with $3\mid (q-\epsilon)$.  Then there exist nontrivial $\chi_1, \chi_2, \chi_3 \in\irrp 3{B_0(S)}^{\sigma_1}$ such that $\chi_1$ and $\chi_2$ extend to $\aut{S}$. 
\end{pro}
\begin{proof}
In this case, we see that all three unipotent characters are members of $\irrp 3{\wt{G}}$ and that there is a unique unipotent block $B_0(\wt{G})$.  Further, the unipotent characters are rational-valued, and therefore are members of $\irrp 3{B_0(\wt{G})}^{\sigma_1}$.  Then we may take $\chi_1$ and $\chi_2$ to be the restrictions to $G$ (viewed as a character of $S$) of the two nontrivial unipotent characters.  

The semisimple element $s\in\wt{G}^\ast$ with eigenvalues $\{\zeta, \zeta^{-1}, 1\}$, where $\zeta\in\mathbb{F}_{q^2}^\times$ has order $3$, is now conjugate to $sz$ where $z=\zeta\cdot I_3\in \zent{\wt{G}^\ast}$.  The corresponding semisimple character has degree $\chi_s(1)=(q+\epsilon)(q^2+\epsilon q+1)$, so $\chi_s\in\irr{B_0(\wt{G})}^{\sigma_1}$ satisfies $3\mid\mid \chi_s(1)$ and is not irreducible on restriction to $G$.  Then the constituents of the restriction to $G$ are members of $\irrp 3{B_0({G})}$, and are trivial on $\zent{G}$ since $s\in [\wt{G}^\ast, \wt{G}^\ast]\cong G$, so it suffices to see that they are also $\sigma_1$-invariant, using the character table available in CHEVIE.
\end{proof}

Together, Propositions \ref{lietypecross} through \ref{SL3notcyclic} yield Theorem \ref{simples3} for the simple groups of Lie type in non-defining characteristic, completing parts (i) and (ii).

\subsubsection{Lie Type in Defining Characteristic for $p=3$}

We now consider the case $S$ is as in Section \ref{sec:setup} with $\bG$ of simply connected type defined in characteristic $3$.  Let $(\bG^\ast, F^\ast)$ be dual to $(\bG, F)$.  Keep in mind the notations and considerations of Section \ref{generalities}, where now $q$ is a power of $3$. Note that $|\wt{S}/S|=|\zent{G}|$, and this is $1$ unless $G$ is of classical type or $G=E_{7, sc}(q)$. 

Since $\wt{G}/G$ has size prime to $3$, it follows that any irreducible character of $G$ lying under $\irrp 3 {\wt{G}}$ is a member of $\irrp 3{G}$.  Since $|\sigma_1|$ is a power of $3$, we further see that for any $\wt{\chi}\in\irrp 3{\wt{G}}^{\sigma_1}$, there is a member of $\irrp 3{G}^{\sigma_1}$ lying under $\wt{\chi}$.  We also have $\irrp 3{B_0(S)}=\irrp 3{S}$, so any member of $\irrp 3{G}$ with $\zent{G}$ in its kernel may be viewed as a member of $\irrp 3{B_0(S)}$.

 Now, given a semisimple element $s\in\wt{G}^\ast$, we have $|s|$ is prime to $3$, and hence the Lusztig series $\mathcal{E}(\wt{G}, s)$ is fixed by $\sigma_1$ using Lemma \ref{SFTlem3.4}.  Then in particular, the unique semisimple character $\wt{\chi}_s\in\irrp 3 {\wt{G}}$ in this series must be fixed by $\sigma_1$.    To illustrate three nontrivial characters of $\irrp 3{G}^{\sigma_1}$ that are not $\aut S$-conjugate, it therefore suffices to show that there are semisimple elements $1\neq s_1, s_2, s_3 \in\wt{G}^\ast$ such that 
  \begin{itemize}

  \item[(1)]  $s_i$ is not $\wt{G}^\ast$-conjugate to $s_j^\varphi z$ for $i\neq j$, $z\in \zent{\wt{G}^\ast}$, and $\varphi$ any (possibly trivial) graph-field automorphism.
 \end{itemize}
 
 In most cases, we further ensure that one of these characters is $\aut{S}$-invariant, by choosing $s_1$ so that 
 \begin{itemize}
 \item[(2)] the class of $s_1$ is invariant under graph-field automorphisms and 
 \item[(3)] $s_1$ is not $\wt{G}^\ast$-conjugate to $s_1z$ for any $1\neq z\in \zent{\wt{G}^\ast}$.
 \end{itemize}
Property (2) will ensure that $\wt{\chi}_{s_1}$ is invariant under graph-field automorphisms, using \cite[Corollary 2.4]{NTTcyclo}, and property (3) will imply that $\wt{\chi}_{s_1}$ restricts irreducibly to $G$, so the resulting character of $G$ is $\aut S$- and $\sigma_1$-invariant. 
 Finally, we will choose $s_1, s_2,$ and $s_3$ such that 
  \begin{itemize}
 \item[(4)] $s_i\in [\wt{G}^\ast, \wt{G}^\ast]$ for $i=1,2,3$,
 \end{itemize}
  so that the $\wt{\chi}_{s_i}$ are trivial on $\zent{\wt{G}}$,   
  ensuring that all three characters of $G$ may be viewed as characters of $\irrp 3{B_0(S)}^{\sigma_1}$ from the above discussion.

\begin{pro}
Let $S=G_2(q)$, $\tw{3}D_4(q)$, $\tw{2}G_2(q)$, $E_6^\pm(q)$, $E_7(q)$, $F_4(q)$ or $E_8(q)$ be simple with $q$ a power of $3$.  Then there exist nontrivial $\chi_1, \chi_2, \chi_3\in\irrp 3{B_0(S)}^{\sigma_1}$ that are pairwise not $\aut{S}$-conjugate and such that $\chi_1$ is $\aut{S}$-invariant. 
\end{pro}
\begin{proof}
Note that we may assume $S$ is not one of the groups from Proposition \ref{sporadicalt3}.  The character degrees in these cases are available at \cite{luebeckwebsite}.  If $S$ is $G_2(q)$, $\tw{3}D_4(q)$, $\tw{2}G_2(q)$, $E_6(q)$, $\tw{2}E_6(q)$, $F_4(q)$ or $E_8(q)$, then $\wt{G}=G=S$ and there is a nontrivial odd character degree of multiplicity one, which therefore must be  $\sigma_1$- and $\aut{S}$-invariant.  Similarly, $E_7(q)$ has a unique character of degree $\Phi_3\Phi_6\Phi_7\Phi_9\Phi_{12}\Phi_{14}\Phi_{18}$, which restricts irreducibly from a character of $\wt{S}=E_7(q)_{ad}$.   
Finally, in each case there are at least two more semisimple characters with different degrees, which must yield members of $\irrp 3{B_0(S)}^{\sigma_1}$ by the above discussion.  
\end{proof}

\begin{pro}\label{typeAdefining}
Let $S=PSL_n^\epsilon(q)$ be simple with $q$ a power of $3$ and $n\geq 2$ and let $X\leq \aut{S}$ such that $X/S$ is a Sylow $3$-subgroup of $\aut{S}/S$.  Then there exist nontrivial $\chi_1, \chi_2, \chi_3\in\irrp 3{B_0(S)}^{\sigma_1}$ that are pairwise not $\aut{S}$-conjugate and such that $\chi_1$ is $X$-invariant. 
Further, if $n\geq 3$, then $\chi_1$ may be chosen to be $\aut{S}$-invariant, and if $n\geq 5$, then $\chi_1, \chi_2, \chi_3$ may all be chosen to be $\aut{S}$-invariant. 
\end{pro}
\begin{proof}
Throughout, let $\delta\in\mathbb{F}_{q^2}^\times$ have order $4$ and assume $S$ is not isomorphic to one of the groups in Proposition \ref{sporadicalt3}.  Recall that the conjugacy classes of semisimple elements in $\wt{G}^\ast=GL_n^\epsilon(q)$ are determined by their eigenvalues and that $\zent{\wt{G}^\ast}$ is comprised of scalar matrices.

 If $n=2$, then $|\wt{S}/S|=2$ and $\aut{S}/\wt{S}$ is generated by a field automorphism.  The semisimple elements $s_1, s_2,$ and $s_3$ with eigenvalues $\{\delta, \delta^{-1}\}$, $\{\zeta_1, \zeta_1^{-1}\}$, and $\{\xi_1, \xi_1^{-1}\}$ with $\zeta_1\in\mathbb{F}_q^\times$ and $\xi_1\in\mathbb{F}_{q^2}^\times\setminus\mathbb{F}_{q}^\times$ and $|\xi|\neq 4\neq |\zeta|$ satisfy properties (1), (2), and (4).  Now, since $\wt{\chi}_{s_1}$ is fixed by field automorphisms, and hence by $X$, and since $|\wt{S}/S|$ is relatively prime to $3$, we see that the irreducible constituents of the restriction $(\wt{\chi}_{s_1})_{G}$ are still fixed by $X$ and by $\sigma_1$.  If $n=3$ or $4$, then $s_1, s_2, s_3$ satisfy (1)-(4) if chosen to have eigenvalues $\{\delta, \delta^{-1}\}$, $\{-1, -1\}$, and $\{\xi, \xi^{-1}\}$ with remaining eigenvalues $1$, where $|\xi|>2$ divides $q+\eta$ if $4|q-\eta$.  
  
Now suppose that $n\geq 5$.  Consider semisimple elements $s_1, s_2,$ and $s_3$ of $\wt{G}^\ast=GL_n^\epsilon(q)$ with eigenvalues $(\delta, \delta^{-1}, 1, \ldots, 1)$, $(-1, -1, 1, \ldots, 1)$, and $(\delta, \delta^{-1}, \delta, \delta^{-1}, 1, \ldots, 1)$, respectively.  If $n=6$, instead define $s_3$ to have eigenvalues $(-1, -1, -1, -1, 1, 1)$.  Then these satisfy (1)-(4), and in fact properties (2) and (3) are held by all three elements.  Hence the corresponding semisimple characters $\wt{\chi}_{s_i}$ of $\wt{G}$ are invariant under graph-field automorphisms and restrict irreducibly to members of $\irrp 3{B_0(G)}^{\sigma_1}$ that are trivial on $\zent{G}$.  Hence these restrictions are members of $\irrp 3{B_0(S)}^{\sigma_1}$ invariant under $\aut{S}$. 
\end{proof}

\begin{pro}\label{classicaldefining}
Let $q$ be a power of $3$.  Let $S=PSp_{2n}(q), P\Omega_{2n+1}(q)$, or $P\Omega_{2n}^\pm(q)$ be simple with $n\geq 2, 3, 4$ respectively. 
Then there exist nontrivial $\chi_1, \chi_2, \chi_3\in\irrp 3{B_0(S)}^{\sigma_1}$ that are pairwise not $\aut{S}$-conjugate and such that $\chi_1$ is invariant under $\aut{S}$.  
\end{pro}

\begin{proof}
We may again assume $S$ is not one of the groups in Proposition \ref{sporadicalt3}. Let $\delta\in\mathbb{F}_{q^2}^\times$ with $|\delta|=4$, and we keep the notation from Section \ref{generalities}.  
Let $\Phi$ and $\Delta:=\{\alpha_1, \alpha_2, \ldots, \alpha_n\}$ be a system of roots and simple roots, respectively, for $\wt{\bG}^\ast$ with respect to a maximal torus $\wt{\textbf{T}}^\ast$, following the standard model described in \cite[Remark 1.8.8]{GLS3}. Then $\Phi$ is type $B_n$, $C_n$, or $D_n$ in the case $S=PSp_{2n}(q), P\Omega_{2n+1}(q)$, or $P\Omega_{2n}^\pm(q)$, respectively.  Further, $\Phi$ has no nontrivial graph automorphism unless we are in the case of $D_n$, in which case all members of $\Delta$ have the same length and that automorphism has order $2$ unless $n=4$.  

We use the notation as in \cite{GLS3} for the Chevalley generators.   In particular, given $\alpha\in \Phi$, let $h_\alpha$ denote the corresponding coroot.  Let $\bK:=[\wt{\bG}^\ast, \wt{\bG}^\ast]$, so we have $h_\alpha(t)\in \bK$ for $t\in \overline{\mathbb{F}}_q^\times$ by \cite[Theorem 1.10.1(a)]{GLS3} and $\wt{\bG}^\ast=\bK.\zent{\wt{\bG}^\ast}$.  Notice that for $s, s'\in \bK$ (not necessarily distinct), we have $s$ is $\wt{\bG}^\ast$-conjugate to $s'z$ for $z\in \zent{\wt{\bG}^\ast}$ if and only if $z\in \zent{\bK}$ and the conjugating element can be chosen in $\bK$.

 By \cite[Theorem 1.12.4]{GLS3} and \cite[15.1]{CE04}, $\bK$ is isomorphic as an abstract group to the simply connected simple algebraic group $(\wt{\bG}^\ast)_{sc}$ associated to $\wt{\bG}^\ast$, and the Chevalley relations and generators of $(\wt{\bG}^\ast)_{sc}$  and $\bK$ may be identified.  We will make this identification.  In particular, choosing $s_1, s_2$, and $s_3$ in $\bK$, the properties (1)-(3) may be verified by computation in $\bK$ rather than $\wt{\bG}^\ast$.

Let $\textbf{T}$ denote a maximal torus of $\bK$ under this identification, and note that $\textbf{T}=\langle h_{\alpha}(t)\mid t\in\overline{\mathbb{F}}_q^\times, \alpha\in\Phi\rangle$, and $\norm{\bK}{\textbf{T}}=\langle \textbf{T}, n_\alpha(1)\mid \alpha\in\Phi\rangle$. Further, note that $\textbf{W}:= \norm{\wt{\bG}^\ast}{\wt{\textbf{T}}^\ast}/\wt{\textbf{T}}^\ast\cong \norm{\bK}{\textbf{T}}/\textbf{T} $. By \cite[Cor. 0.12]{dignemichel}, we know that $\norm \bK{\textbf{T}}$ controls fusion in $\textbf{T}$, so two elements of $\textbf{T}$ are conjugate if and only if there is a conjugating element in $\textbf{W}$.  Further, we have an isomorphism $(\overline{\mathbb{F}}_q^\times)^n\rightarrow \textbf{T}$ given by $(t_1,\ldots, t_n)\mapsto \prod_{i=1}^n h_{\alpha_i}(t_i)$. 

Now using the standard model for $\Phi$ and $\Delta$ as in \cite{GLS3}, since $\Phi$ is type $B_n, C_n,$ or $D_n$, we have $\alpha_i:=e_i-e_{i+1}$ for $1\leq i\leq n-1$, where $\{e_1,\ldots, e_n\}$ is an orthonormal basis for the $n$-dimensional Euclidean space.  Here $\textbf{W}\leq C_2\wr \mathfrak{S}_n$ where the generators of the base subgroup $C_2^n$ act via negation on the $e_i$'s and the copy of $\mathfrak{S}_n$ permutes the $e_i$'s.  

Using this information and the description of $\zent{\bK}$ in \cite[Table 1.12.6]{GLS3}, computation with the Chevalley relations yields that the element $s_1':=h_{\alpha_1}(\delta)$ is not $\wt{\bG}^\ast$-conjugate to $s_1'z$ for any $1\neq z\in \zent{\wt{\bG}^\ast}$.  If $\delta\in\mathbb{F}_q^\times$, we see that $s_1'$ is $F^\ast$-fixed, and we write $s_1:=s_1'$.  Otherwise, let $\dot{s}_{\alpha_1}\in \textbf{W}$ induce the reflection corresponding to $\alpha_1$.  Then $s_1:=s_1'^g$ is $F^\ast$-fixed, where $g\in \wt{\bG}^\ast$ satisfies $g^{-1}F^\ast(g)=\dot{s}_\alpha$.  (Note that such a $g$ exists by the Lang-Steinberg theorem.)  

Let $F_3$ denote a generating field automorphism such that $F_3(h_\alpha(t))=h_\alpha(t^3)$ for $\alpha\in\Phi$ and $t\in\overline{\mathbb{F}}_q^\times$. Then $s_1'$ is $\wt{\bG}^\ast$-conjugate to $F_3(s_1')$, taking for example $\dot{s}_{\alpha_1}$ as the conjugating element.  Hence $s_1$ is also $\wt{\bG}^\ast$-conjugate to $F_3(s_1)$. Since the $\cent{\bG^\ast}{\iota^\ast(s_1)}$ is connected, using \cite[Corollary 2.8(a)]{bonnafe05}, this yields that the $\wt{G}^\ast$-conjugacy class of $s_1$ is fixed by field automorphisms, using \cite[(3.25)]{dignemichel}. Further, by construction, the $\wt{G}^\ast$-conjugacy class of $s_1$ is fixed by graph automorphisms unless $\Phi$ is type $D_4$.  In the latter case, we may make similar considerations using $s_1':=h_{\alpha_2}(\delta)$.

Now, further taking $s_2:=h_{\alpha_1}(-1)$ and $s_3\in \bK^{F^\ast}$ an element of order larger than $4$, we obtain properties (1)-(4).  
\end{proof}

Theorem \ref{simples3} now follows from Propositions \ref{sporadicalt3} -- \ref{classicaldefining}, completing the proof of Theorem A.